\documentclass[12pt]{amsart}
\usepackage[utf8]{inputenc}
\usepackage{geometry, amsmath, amssymb, color}
\usepackage{tikz, verbatim}
\usetikzlibrary{decorations.pathmorphing}

\newtheorem{theorem}{Theorem}[section]
\newtheorem{lemma}[theorem]{Lemma}
\newtheorem{proposition}[theorem]{Proposition}
\newtheorem{corollary}[theorem]{Corollary}
\theoremstyle{definition}
\newtheorem{definition}[theorem]{Definition}
\newtheorem{remark}[theorem]{Remark}
\newtheorem{example}[theorem]{Example}

\definecolor{Cerulean}{cmyk}{0.94,0.11,0,0}
\definecolor{LimeGreen}{cmyk}{0.50,0,1,0}
\definecolor{NavyBlue}{cmyk}{0.94,0.54,0,0}
\definecolor{RawSienna}{cmyk}{0,0.72,1,0.45}

\definecolor{applegreen}{rgb}{0.55,0.71,0.0}

\definecolor{darkblue}{rgb}{0.0, 0.0, 1}

\newcommand{\bbZ}{\mathbb{Z}}
\newcommand{\calA}{\mathcal{A}}
\newcommand{\calLP}{\mathcal{LP}}
\newcommand{\calP}{\mathcal{P}}

\newcommand{\calT}{\mathcal{T}}
\newcommand{\ttCat}{\mathtt{Cat}}
\newcommand{\ttNar}{\mathtt{Nar}}
\newcommand{\ttSch}{\mathtt{Sch}}
\newcommand{\ttsch}{\mathtt{sch}}
\newcommand{\inv}{^{-1}}

\tikzstyle{vertex}=[circle, draw=black, inner sep=0pt, minimum size=5pt]
\newcommand{\vertex}{\node[vertex]}

\title{Schr\"oder combinatorics and $\nu$-associahedra}
\author[von Bell]{Matias von Bell}
\author[Yip]{Martha Yip}
\email[von Bell]{matias.vonbell@uky.edu}
\email[Yip]{martha.yip@uky.edu}
\address{Department of Mathematics, University of Kentucky, Lexington KY 40506}
\date{\today}
\begin{document}
\maketitle

\begin{abstract}
We study $\nu$-Schr\"oder paths, which are Schr\"oder paths which stay weakly above a given lattice path $\nu$. Some classical bijective and enumerative results are extended to the $\nu$-setting, including the relationship between small and large Schr\"oder paths. We introduce two posets of $\nu$-Schr\"oder objects, namely $\nu$-Schr\"oder paths and trees, and show that they are isomorphic to the face poset of the $\nu$-associahedron $A_\nu$ introduced by Ceballos, Padrol and Sarmiento. A consequence of our results is that the $i$-dimensional faces of $A_\nu$ are indexed by $\nu$-Schr\"oder paths with $i$ diagonal steps, and we obtain a closed-form expression for these Schr\"oder numbers in the special case when $\nu$ is a `rational' lattice path. Using our new description of the face poset of $A_\nu$, we apply discrete Morse theory to show that $A_\nu$ is contractible. This yields one of two proofs presented for the fact that the Euler characteristic of $A_\nu$ is one.  A second proof of this is obtained via a formula for the $\nu$-Narayana polynomial in terms of $\nu$-Schr\"oder numbers.
\vspace{0.5cm}

\end{abstract}

\section{Introduction}
The $n$-dimensional {\em associahedron} is a simple polytope whose face poset is isomorphic to the poset of diagonal dissections of a convex $(n+3)$-gon ordered by coarsening, so that the minimal elements of the poset are the triangulations of the $(n+3)$-gon and the maximal element is the empty dissection.
A well-known equivalent statement is that the face poset of the $n$-associahedron is isomorphic to the poset of Schr\"oder paths from $(0,0)$ to $(n+1,n+1)$.
The number of elements in the poset of Schr\"oder paths is known as a {\em Schr\"oder--Hipparchus number} or a {\em small Schr\"oder number}.
This poset is graded by the number of diagonal steps in a Schr\"oder path, so the number of Schr\"oder paths with $i$ diagonal steps is the number of faces of the associahedron of dimension $i$.  
In particular, the vertices of the $n$-associahedron correspond to the Schr\"oder paths with no diagonal steps, which are better known as {\em Dyck paths}.
The {\em Tamari lattice} is a partial order on the set of Dyck paths, and a notable result~\cite{Hai84, Lee89} is that its Hasse diagram is realizable as the $1$-skeleton of the $n$-associahedron.

With the viewpoint that the set of Dyck paths is the set of lattice paths that lie weakly above the staircase path $(NE)^{n+1}$ from $(0,0)$ to $(n+1,n+1)$, Pr\'eville-Ratelle and Viennot~\cite{PV17} extended the notion of the Tamari lattice to the $\nu$-Tamari lattice, a partial order on the set of {\em $\nu$-Dyck paths}, which are lattice paths that lie weakly above a fixed lattice path $\nu$ from $(0,0)$ to $(b,a)$.
A striking result of Ceballos, Padrol and Sarmiento~\cite[Theorem 5.2]{CPS19} is that the Hasse diagram of the $\nu$-Tamari lattice is realizable as the $1$-skeleton of a polyhedral complex induced by an arrangement of tropical hyperplanes. 
This polyhedral complex is the {\em $\nu$-associahedron} $A_\nu$.

Inspired by the connection between Schr\"oder paths and faces of the associahedron, we define $\nu$-Schr\"oder paths (see Definition~\ref{defn.nuschroderpath}) and make a similar connection to the faces of the $\nu$-associahedron.
In this article we also consider another family of Schr\"oder objects: $\nu$-Schr\"oder trees (see Definition~\ref{defn.nutrees}) are a generalization of the $\nu$-trees of Ceballos, Padrol and Sarmiento~\cite{CPS20}. 

The face poset of $A_\nu$ is defined in~\cite{CPS19} to be the poset on {\em covering $(I,\overline{J})$-forests}, which may be considered as a $\nu$-analogue of non-crossing partitions. 
A central result of this article is the following pair of alternative descriptions of the face poset of $A_\nu$.

\noindent{\bf Theorem~\ref{thm:isoposets}.} 
The face poset of the $\nu$-associahedron is isomorphic to the poset on $\nu$-Schr\"oder paths, and to the poset on $\nu$-Schr\"oder trees.

\noindent{\bf Corollary~\ref{cor:isoposets}.}
The number of $i$-dimensional faces of the $\nu$-associahedron is the number of $\nu$-Schr\"oder paths with $i$ diagonal steps.

This is the outcome of combining bijections between covering $(I,\overline{J})$-trees, $\nu$-Schr\"oder trees, and $\nu$-Schr\"oder paths from {\bf Theorem~\ref{thm.coveringforesttreebijection}} and {\bf Theorem~\ref{treePosetIsPathPoset}}, and showing how the cover relation on the covering $(I,\overline{J})$-trees is translated to cover relations for the other $\nu$-Schr\"oder objects. 
Figures~\ref{fig:53poset} and~\ref{fig.35associahedron} show an example of a $\nu$-associahedron and its face poset in terms of $\nu$-Schr\"oder paths.

Combining Theorem~\ref{thm:isoposets} with the fact that $A_\nu$ is a polyhedral complex, we conclude in {\bf Theorem~\ref{thm:lattice}} that if $P_\nu$ is the contraction poset of $\nu$-Schr\"oder paths or the poset of $\nu$-Schr\"oder trees, then every interval in $P_\nu$ is an Eulerian lattice. As a consequece of results by Pr\'eville-Ratelle and Viennot \cite{PV17} we conclude in {\bf Corollary~\ref{cor:sublattice}} that for non-classical $\nu$, adjoining a minimal and a maximal element to $P_\nu$ yields a sublattice in the face lattice of a classical associahedron. 

From Theorem~\ref{thm:isoposets}, it also follows that the Euler characteristic of the $\nu$-associahedron is $\chi(A_\nu)=\sum_{i\geq0} (-1)^i \ttsch_\nu(i)$, where $\ttsch_\nu(i)$ is the number of $\nu$-Schr\"oder paths with $i$ diagonal steps.
We present two proofs of the fact that $\chi(A_\nu)=1$; one enumerative and one topological.
The enumerative proof relies on the {\em $\nu$-Narayana polynomial} $N_\nu(x)$ (see Definition~\ref{defn.narayana}), which is a generating function for $\nu$-Dyck paths with respect to their {\em valleys}.  
In {\bf Proposition~\ref{prop.narayana}}, we show that $N_\nu(x+1)=\sum_{i\geq0}\ttsch_\nu(i)x^i$, from which it follows that $\chi(A_\nu) = N_\nu(0)=1$, as there is a unique $\nu$-Dyck path with zero valleys.

From Proposition~\ref{prop.narayana}, it can also be deduced that the number of large $\nu$-Schr\"oder paths is twice the number of (small) $\nu$-Schr\"oder paths. 
A bijective proof of this fact is also given.
This generalizes results of Aguiar and Moreira~\cite{AM06} and Gessel~\cite{Ges09}.

The topological proof that $\chi(A_\nu)=1$ employs discrete Morse theory.
We show in {\bf Theorem~\ref{thm:contractible}} that the contraction poset of $\nu$-Schr\"oder paths has an acyclic matching with a unique critical element, thereby showing that the $\nu$-associahedron is contractible. The matching has a simple description in terms of $\nu$-Schr\"oder paths, highlighting a benefit of this viewpoint.

Another avenue for generalizing the theory of Schr\"oder paths is to define $(q,t)$-analogues.
Haglund~\cite[Section 4]{Hag08} developed the theory of the $(q,t)$-Schr\"oder polynomial in connection with the theory of Macdonald polynomials and diagonal harmonics.
Song~\cite{Son05}, and Aval and Bergeron~\cite{AB18} further extended this $(q,t)$ generalization to the case of $(a,ma+1)$- and $(a,b)$-Schr\"oder paths for any positive integers $a,b$.
The idea of Schr\"oder parking functions was also explored in~\cite{AB18}. 
We anticipate that the $(q,t)$-analogue can be further extended to the case of $\nu$-Schr\"oder paths.

This article is organized as follows. 
In Section~\ref{sec.enumSchroder}, a number of existing bijective and enumerative results on Schr\"oder paths are extended to the case of $\nu$-Schr\"oder paths.
Closed-form expressions are obtained for Schr\"oder numbers with respect to the number of diagonal step in the special case when $\nu$ is a `rational' lattice path. 
In Section~\ref{sec.SchroderTrees}, a bijection between $\nu$-Schr\"oder trees and paths is given.
Furthermore, a poset structure on the set of $\nu$-Schr\"oder trees given by contraction operations is defined, and it is shown that this induces a poset structure on the set of $\nu$-Schr\"oder paths.
In Section~\ref{sec.nuAssociahedron}, the face poset of the $\nu$-associahedron is shown to have alternative descriptions in terms of $\nu$-Schr\"oder trees and paths. 
An acyclic partial matching on the Hasse diagram of the contraction poset of $\nu$-Schr\"oder paths is exhibited, giving a proof that the $\nu$-associahedron is contractible. 

\thanks{{\bf Acknowledgments.}
We thank Richard Ehrenborg for suggesting the use of Discrete Morse Theory to show the contractibility of the $\nu$-associahedron. 
MY was partially supported by the Simons Collaboration Grant 429920.
}

\section{Small and large Schr\"oder paths}\label{sec.enumSchroder}
We begin by presenting some preliminary definitions of the various kinds of lattice paths that we will consider.
\begin{definition}\label{defn.nuschroderpath}
A {\em lattice path} in the rectangle defined by $(0,0)$ and $(b,a)$ in $\bbZ_{\geq 0}^2$ is a sequence of $a$ north steps $N=(0,1)$ and $b$ east steps $E=(1,0)$.

Let $\nu$ be a lattice path in the rectangle defined by $(0,0)$ and $(b,a)$.
A {\em $\nu$-Dyck path} is a lattice path from $(0,0)$ to $(b,a)$ which stays weakly above the path $\nu$.

A {\em peak} of $\nu$ is a consecutive $NE$ pair in $\nu$, and a {\em high peak} is a peak that occurs strictly above the path $\nu$. 
A {\em valley} of $\nu$ is a consecutive $EN$ pair in $\nu$. 
The {\em $\nu$-diagonal} is defined to be the set of squares immediately below the peaks of $\nu$. 
Let $\mu$ denote the path obtained from $\nu$ by replacing each of its peaks with a diagonal $D=(1,1)$ step. 
A {\em (small) $\nu$-Schr\"oder path} is a Schr\"oder path from $(0,0)$ to $(b,a)$ which stays weakly above the path $\nu$. 
A {\em large $\nu$-Schr\"oder path} is a a Schr\"oder path from $(0,0)$ to $(b,a)$ which stays weakly above the path $\mu$.
Let $\calP_\nu$ and $\calLP_\nu$ denote the set of small and large $\nu$-Schr\"oder paths respectively. 
Figure~\ref{fig:nuSchPaths} provides some examples of both small and large $\nu$-Schr\"oder paths. 
For a small $\nu$-Schr\"oder path $\pi$ we define its area, denoted by $\mathrm{area}(\pi)$, to be the area of the region between $\pi$ and $\nu$. 
For example, the small $\nu$-Schr\"oder path on the right in Figure~\ref{fig:nuSchPaths} has area $1.5$.
\end{definition}

\begin{figure}[ht!]
\begin{tikzpicture}[scale=0.4]
\begin{scope}[xshift=0]
    \draw[help lines] (0,0) grid (6,3);
    
    \draw[fill=gray, opacity=0.25] (0,0) rectangle (6,1);
    \draw[fill=gray, opacity=0.25] (2,1) rectangle (6,2);
    \draw[fill=gray, opacity=0.25] (4,2) rectangle (6,3);    
    
    \draw[fill=gray, opacity=0.5] (0,0) rectangle (1,1);
    \draw[fill=gray, opacity=0.5] (2,1) rectangle (3,2);
    \draw[fill=gray, opacity=0.5] (4,2) rectangle (5,3);	
	\draw[dashed, thick, white] (0,0) -- (6,3);    
    \draw[very thick,red] (0,0) -- (1,1) -- (2,1) -- (3,2) -- (3,3) --  (6,3);
    
    \node at (-.8,0) {\tiny$(0,0)$};
    \node at (6.8,3) {\tiny$(6,3)$};
\end{scope}
\begin{scope}[xshift=320]
    \draw[help lines] (0,0) grid (5,4);
    \draw[fill=gray, opacity=0.25] (1,0) rectangle (5,1); 
    \draw[fill=gray, opacity=0.25] (3,1) rectangle (5,4);
    
    \draw[fill=gray, opacity=0.5] (1,0) rectangle (2,1);
    \draw[fill=gray, opacity=0.5] (3,3) rectangle (4,4);
    
	\draw[very thick,red] (0,0) -- (1,0) -- (1,2) -- (2,3) -- (3,3) -- (4,4) -- (5,4);
	
	\node at (-.8,0) {\tiny$(0,0)$};
    \node at (5.8,4) {\tiny$(5,4)$};
\end{scope}
\begin{scope}[xshift=640]
    \draw[help lines] (0,0) grid (5,4);
    \draw[fill=gray, opacity=0.25] (0,0) rectangle (5,1); 
    \draw[fill=gray, opacity=0.25] (1,1) rectangle (5,2); 
    \draw[fill=gray, opacity=0.25] (2,2) rectangle (5,3);
    \draw[fill=gray, opacity=0.25] (3,3) rectangle (5,4);
    
    \draw[fill=gray, opacity=0.5] (0,0) rectangle (1,1);
    \draw[fill=gray, opacity=0.5] (1,1) rectangle (2,2);
    \draw[fill=gray, opacity=0.5] (2,2) rectangle (3,3);
    \draw[fill=gray, opacity=0.5] (3,3) rectangle (4,4);    
    \draw[dashed, thick, white] (0,0) -- (5,4);
	\draw[very thick,red] (0,0) -- (0,2) -- (1,2) -- (2,3) -- (3,3) -- (3,4) -- (5,4);
    
    \node at (-.8,0) {\tiny$(0,0)$};
    \node at (5.8,4) {\tiny$(5,4)$};
\end{scope}
\end{tikzpicture}
\caption{From left to right: a large $\nu$-Schr\"oder path with $\nu=(NEE)^3$, a large $\nu$-Schr\"oder path with $\nu = ENEENNNEE$, and a (small) $(4,5)$-Schr\"oder path. The region below $\nu$ is shaded in gray, with the $\nu$-diagonal in a darker gray.}
\label{fig:nuSchPaths}
\end{figure}
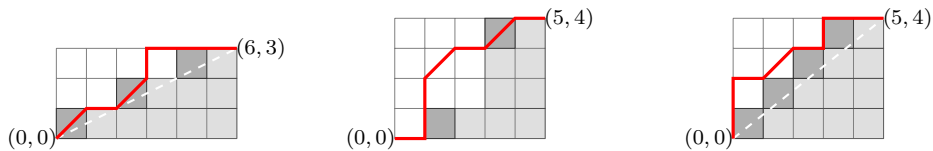

\begin{remark}
We point out two special cases; the rational $(a,b)$ case, and the classical case.

First, the line segment from $(0,0)$ to $(b,a)$ determines a unique lowest lattice path that stays weakly above it, that is, the unique lattice path $\nu=\nu(a,b)$ with valleys at the lattice points $\{ \hbox{$(k, \lceil \frac{ka}{b} \rceil)$} \,\big|\,  \lceil \frac{ka}{b} \rceil \neq \lceil \frac{(k+1)a}{b} \rceil,  1\leq k\leq b-1\}$. 
When $a$ and $b$ are coprime with $a<b$, the set of $\nu$-Dyck paths is the set of `rational' $(a,b)$-Dyck paths defined in by Armstrong, Rhoades and Williams~\cite{ARW13}.
The lattice path on the right in Figure~\ref{fig:nuSchPaths} is an example of a rational $(4,5)$-Schr\"oder path where $\nu=\nu(4,5)$ is determined by the white dotted line segment from $(0,0)$ to $(5,4)$.  
We point out that the lattice path $\nu$ on the left in Figure~\ref{fig:nuSchPaths} is also determined by the line segment from $(0,0)$ to $(b,a)$, but we do not consider this to be a rational case as $a=3$ and $b=6$ are not coprime.

Furthermore, in the case $a=n$ and $b=n+1$ for some positive integer $n$, the path $\nu(n,n+1)=(NE)^nE$, and the set of $\nu$-Dyck paths is equivalent to the set of `classical' Dyck paths, which are often defined as lattice paths from $(0,0)$ to $(n,n)$ that do not fall below the line $y=x$.
\end{remark}

Aguiar and Moreira~\cite[Proposition 3.1]{AM06} showed that the set of classical large Schr\"oder paths can be partitioned into two halves where one half consists of paths that do not contain $D$ steps on the diagonal, and the other half consists of paths that contain at least one $D$ step on the diagonal. 
Gessel~\cite{Ges09} showed that the same result holds in the more general rational $(a,b)$-case. 
We further generalize Gessel's argument to the setting of $\nu$-Schr\"oder paths. 

\begin{theorem} \label{smallhalfofbig} 
Let $\nu$ be a lattice path. 
Then $|\calLP_\nu| =2\,|\calP_\nu|$ if and only if $\nu$ begins with a north step and ends with an east step. 
\end{theorem}
\begin{proof}
Suppose $\nu$ begins with a north step and ends with an east step. 
Let $\calA=\calLP_\nu\backslash\calP_\nu$ denote the set of $\nu$-Schr\"oder paths with at least one $D$ step on the $\nu$-diagonal. 
Define a map $f:\calP_\nu \to \calA$ as follows:
A path $\mu \in \calP_\nu$ can be partitioned as $N\mu_1E\mu_2$, where $E$ is the first $E$ step on the $\nu$-diagonal. 
The existence of such an $E$ step is guaranteed by the fact that $\nu$ ends in an $E$ step, so there is a $\nu$-diagonal square in the top row. 
Let $f(\mu)$ be the path $\mu_1D\mu_2$.  
We claim that $f$ is a bijection.

To see that $f(\mu)\in \calA$, note that $f$ shifts the steps in $\mu_1$ down by one unit, while the steps in $\mu_2$ remain fixed. 
Thus the $D$ step of $f(\mu)$ which is between $\mu_1$ and $\mu_2$ occurs on the $\nu$-diagonal, since it replaced the $E$ step of $\mu$ which preceeded $\mu_2$.

A step in $N\mu_1$ can only intersect a horizontal run in $\nu$ at the leftmost lattice point of the horizontal run, since otherwise the first $E$ step of the horizontal run is an $E$ step of $\mu$ on the $\nu$-diagonal.
Therefore, only $N$ steps and $D$ steps which intersect only the leftmost lattice points of horizontal runs can occur in $N\mu_1$, both of which remain weakly above $\nu$ after shifting down by one unit. 
Thus $f(\mu)\in \calP_\nu$, and so $f$ is well-defined.

The inverse map $f\inv:\calA\to \calP_\nu$ is defined as follows: For $\pi \in \calA$, partition $\pi$ into $\pi_1D\pi_2$ where $D$ is the last $D$ step on the  $\nu$-diagonal. 
Then $f\inv$ is given by $\pi_1D\pi_2 \mapsto N\pi_1E\pi_2$, with $f(f\inv(\pi))=\pi$ and $f\inv (f(\mu)) =\mu$. 
Hence $f$ is a bijection, and $|\calLP_\nu|= 2\,|\calP_\nu|$. 

Conversely, suppose $\nu$ does not begin with a north step. 
The map $f\inv:\calA\to \calP_\nu$ is injective, but for any path $\rho \in \calP_\nu$ that begins with a $D$ (or $E$) step there is no path $\sigma \in \calA$ such that $f\inv(\rho)=\sigma$. 
Hence $|\calA| < |\calP_\nu|$ and so $2\, |\calP_\nu| \neq |\calLP_\nu|$. 
The case when $\nu$ does not end with an $E$ step can be argued similarly.  
\end{proof}

Recall that a high peak of a lattice path $\nu$ is a peak that occurs strictly above $\nu$. 
A $\nu$-Dyck path is completely determined by its high peaks. It is also completely determined by its valleys. See Figure~\ref{highPeaks} for an example.
\begin{figure}[ht!]
\begin{tikzpicture}[scale=0.5]
\begin{scope}[xshift=0, yshift=0]
    \draw[help lines] (0,0) grid (6,3);
    \draw[fill=gray, opacity=0.25] (0,0) rectangle (6,1); \draw[fill=gray, opacity=0.25] (2,1) rectangle (6,2);
    \draw[fill=gray, opacity=0.25] (4,2) rectangle (6,3);

	\vertex[fill=NavyBlue, minimum size=4pt] at (0,2) {};
	\vertex[fill=NavyBlue, minimum size=4pt] at (0,3) {};
	\vertex[fill=NavyBlue, minimum size=4pt] at (1,2) {};	
	\vertex[fill=NavyBlue, minimum size=4pt] at (1,3) {};
	\vertex[fill=NavyBlue, minimum size=4pt] at (2,3) {};
	\vertex[fill=NavyBlue, minimum size=4pt] at (3,3) {};		

    \node at (3,-1) {Possible high peaks};
\end{scope}
\begin{scope}[xshift=350]
    \draw[help lines] (0,0) grid (6,3);
    \draw[fill=gray, opacity=0.25] (0,0) rectangle (6,1);
    \draw[fill=gray, opacity=0.25] (2,1) rectangle (6,2);
    \draw[fill=gray, opacity=0.25] (4,2) rectangle (6,3);

	\vertex[fill=cyan, minimum size=4pt] at (1,1) {};
	\vertex[fill=cyan, minimum size=4pt] at (1,2) {};
	\vertex[fill=cyan, minimum size=4pt] at (2,1) {};	
	\vertex[fill=cyan, minimum size=4pt] at (2,2) {};
	\vertex[fill=cyan, minimum size=4pt] at (3,2) {};
	\vertex[fill=cyan, minimum size=4pt] at (4,2) {};		
    
    \node at (3,-1) {Possible valleys};
\end{scope}
\begin{scope}[xshift=0, yshift=-180]
    \draw[help lines] (0,0) grid (6,3);

    \draw[fill=gray, opacity=0.25] (0,0) rectangle (6,1);
    \draw[fill=gray, opacity=0.25] (2,1) rectangle (6,2);
    \draw[fill=gray, opacity=0.25] (4,2) rectangle (6,3);
	
	\draw[very thick,red] (0,0) -- (0,2) -- (2,2) -- (2,3) -- (6,3);    
	
	\vertex[fill=NavyBlue, minimum size=4pt] at (0,2) {};
	\vertex[fill=NavyBlue, minimum size=4pt] at (2,3) {};
    \node at (3,-1) {Dyck path determined};
    \node at (3,-2) {by two high peaks};
\end{scope}
\begin{scope}[xshift=350, yshift=-180]
    \draw[help lines] (0,0) grid (6,3);
    
    \draw[fill=gray, opacity=0.25] (0,0) rectangle (6,1);
    \draw[fill=gray, opacity=0.25] (2,1) rectangle (6,2);
    \draw[fill=gray, opacity=0.25] (4,2) rectangle (6,3);

	\draw[very thick,red] (0,0) -- (0,1) -- (1,1) -- (1,2) -- (3,2) -- (3,3) --  (6,3);
    
	\vertex[fill=cyan, minimum size=4pt] at (1,1) {};
    \vertex[fill=cyan, minimum size=4pt] at (3,2) {}; 
	 
    \node at (3,-1) {Dyck path determined}; 
    \node at (3,-2) {by the two corresponding valleys};
\end{scope}    
\end{tikzpicture}
\caption{The possible high peaks and valleys for $\nu = NEENEENEE$ (top row), and a $\nu$-Dyck path determined by a pair of high peaks along with the $\nu$-Dyck path determined by the corresponding pair of valleys.}
\label{highPeaks}
\end{figure}
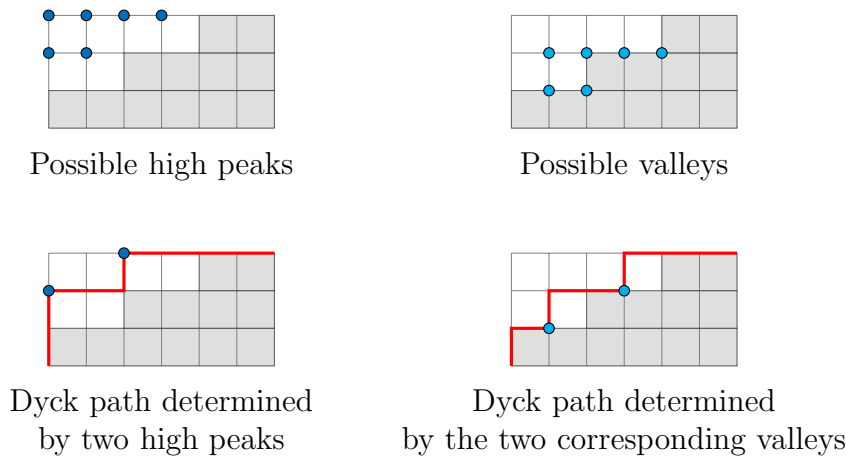

The proof of the next result is a direct generalization of the arguments in Deutsch~\cite{Deu98} and Gessel~\cite{Ges09} to the $\nu$-setting. 



\begin{lemma}\label{lem.highpeaks}
Let $\nu$ be a lattice path that begins with a north step and ends with an east step.
The set of $\nu$-Dyck paths with $i$ high peaks is in bijection with the set of $\nu$-Dyck paths with $i+1$ peaks. 
\end{lemma}
\begin{proof} 
Since $\nu$ is a lattice path that begins with a north step and ends with an east step, then each $\nu$-Dyck path with $i+1$ peaks is determined by its $i$ valleys, and it suffices to show that there is a bijection between the set of $\nu$-Dyck paths with $i$ high peaks and the set of $\nu$-Dyck paths with $i$ valleys. 
A bijection is given by mapping a $\nu$-Dyck path with high peaks at the lattice points $(p_1,q_1),\ldots,(p_i, q_i)$ to the $\nu$-Dyck path with valleys at the lattice points $(p_1+1, q_1-1),\ldots, (p_i+1, q_i-1)$, and mapping the unique $\nu$-Dyck path with no valleys to the unique $\nu$-Dyck path with no high peaks (which is $\nu$ itself). 
This map is well-defined because high peaks are strictly above the path $\nu$.
The inverse map sends a $\nu$-Dyck path with $i$ valleys at the lattice points $(p_1,q_1),\ldots,(p_i,q_i)$ to the $\nu$-Dyck path with $i$ high peaks at $(p_1-1,q_1+1),\ldots,(p_i-1,q_i+1)$, so the map is a bijection.
\end{proof} 

\begin{definition}\label{defn.narayana}
The {\em $i$-th $\nu$-Narayana number} $\mathtt{Nar}_\nu(i)$ is the number of $\nu$-Dyck paths with exactly $i$ valleys.
The $\nu$-{\em Narayana polynomial} is
$$ N_\nu(x) = \sum_{i\geq 0} \mathtt{Nar}_\nu(i) x^i .$$
\end{definition}
This generalization of the Narayana numbers was introduced by Ceballos, Padrol and Sarmiento~\cite{CPS19} as the $h$-vector of the $\nu$-Tamari complex.
The rational $(a,b)$ case also appears in the work of Armstrong, Rhoades and Williams~\cite{ARW13} as the $h$-vector of their rational associahedron. 
Bonin, Shapiro and Simion~\cite{BSS93} considered the Narayana polynomial for the dual associahedron.

\begin{proposition} \label{prop.narayana}
Let $\ttsch_\nu(i)$ denote the number of $\nu$-Schr\"oder paths with $i$ diagonal steps. Then
$$N_\nu(x+1) = \sum_{i\geq 0} \ttsch_\nu(i) x^i.$$
\end{proposition}

\begin{proof}
Note that $|\calP_{\nu}| = |\calP_{N\nu E}|$, that is, appending an $N$ step to the beginning of $\nu$ and an $E$ step to the end of $\nu$ does not change the number of $\nu$-Schr\"oder paths. 
Hence we can assume without loss of generality that $\nu$ begins with an $N$ step and ends with an $E$ step. 
By Lemma~\ref{lem.highpeaks}, $\mathtt{Nar}_\nu(i)$ is also the number of $\nu$-Dyck paths with exactly $i$ high peaks. The result then follows from the computation
$$\sum_{j\geq 0 } \mathtt{Nar}_\nu(j) (x+1)^j 
= \sum_{i\geq 0} \sum_{j\geq 0}  \mathtt{Nar}_\nu(j)\binom{j}{i}x^i 
=\sum_{i\geq 0} \ttsch_\nu(i) x^i,$$
where the last equality follows from the observation that for each $\nu$-Dyck path with $j$ high peaks there are exactly $\binom{j}{i}$ ways to choose which $i$ of the high peaks to replace with a $D$ step. 
\end{proof}

\begin{corollary} \label{narayanaPoly}
The number of $\nu$-Schr\"oder paths is given by specializing $N_\nu(x)$ at $x=2$. 
\end{corollary}
\begin{proof}
The claim follows by noting that $|\calP_\nu| = \sum_{i\geq 0} \ttsch_\nu(i)=N_\nu(2)$. 
An alternative way to see this is to note that $\mathtt{Nar}_\nu(i)$ is the number $\nu$-Dyck paths with $i$ high peaks. 
For each of the high peaks, there are two choices; keep the peak or replace it with a $D$ step. Thus the total number of $\nu$-Schr\"oder paths is
$$|\calP_\nu| = \sum_{i\geq 0} \ttNar_\nu(i) 2^{i} = N_\nu(2).$$ 
\end{proof}

\begin{corollary} \label{eulerChar}
Let $\ttsch_\nu(i)$ denote the number of  $\nu$-Schr\"oder paths with $i$ diagonal steps. Then 
$$ \sum_{i\geq 0} (-1)^i\ttsch_\nu(i) = 1.$$ 
\end{corollary}

\begin{proof}
This follows from the fact that $\sum_{i\geq 0} (-1)^i\ttsch_\nu(i) = N_\nu(0)$, and there is a unique $\nu$-Dyck path with no valleys.
\end{proof}

\begin{remark}
Corollary~\ref{eulerChar} can be obtained topologically from the results in Section~\ref{sec.nuAssociahedron} since $\sum_{i\geq 0} (-1)^i\ttsch_\nu(i)$ is the Euler characteristic of the contractible polyhedral complex known as the $\nu$-associahedron.  See Theorem~\ref{thm:contractible}. 
\end{remark}

\begin{remark}
Theorem~\ref{smallhalfofbig} can be deduced from Corollary~\ref{narayanaPoly} since $\mathtt{Nar}_\nu(i)$ is the number $\nu$-Dyck paths with $i+1$ peaks if and only if $\nu$ begins with a $N$-step and ends with an $E$-step, in which case
$$
    |\calLP_\nu| 
    = \sum_{i\geq 0} \#(\nu\text{-Dyck paths with }i+1 \text{ peaks})\cdot 2^{i+1} 
    = 2\sum_{i\geq 0} \mathtt{Nar}_\nu(i) 2^{i} 
    = 2\,|\calP_\nu|.
$$
\end{remark}

Recall that the we refer to the special case when $\nu=\nu(a,b)$ is the lattice path with valleys at $\{(k,\lceil ka/b\rceil) \mid \lceil ka/b\rceil \neq \lceil (k+1)a/b\rceil \}$ as the `rational' case.
We end this section with some enumerative results for the rational $(a,b)$-Schr\"oder paths, but we first recall some results on the rational $(a,b)$-Dyck paths.

For coprime positive integers $a,b$ the {\em rational $(a,b)$-Catalan number} $\ttCat(a,b)$ is the number of $(a,b)$-Dyck paths, and the {\em rational $(a,b)$-Narayana number} $\ttNar(a,b,i)$ is the number of $(a,b)$-Dyck paths with $i$ peaks.
Armstrong, Rhoades and Williams~\cite{ARW13} showed that 
$$\ttCat(a,b) = \frac{1}{a+b}\binom{a+b}{a} 
= \frac{1}{a}\binom{a+b-1}{b} 
= \frac{1}{b}\binom{a+b-1}{a},$$
and for $i=0,\ldots,a$,
$$\ttNar(a,b,i) = \frac{1}{a}\binom{a}{i}\binom{b-1}{b-i}.$$

We now enumerate $(a,b)$-Schr\"oder paths with respect to the number of diagonal steps.  

\begin{definition}
For coprime positive integers $a,b$, and $i=0,\ldots, a$, let $\ttsch(a,b,i)$ denote the number of (small) $(a,b)$-Schr\"oder paths with $i$ diagonal steps and let $\ttSch(a,b,i)$ denote the number of large $(a,b)$-Schr\"oder paths with $i$ diagonal steps.
\end{definition}

We can give an explicit formula for the numbers $\ttSch(a,b,i)$. 
The proof of the following result closely mirrors the one given by Song~\cite[Theorem 2.1]{Son05}, who studied Schr\"oder paths from $(0,0)$ to $(kn, n)$, which is equivalent to the rational case when $a=n$ and $b=kn+1$. 

\begin{proposition} For coprime positive integers $a,b$, and $i=0,\ldots, a$, 
\begin{align*}
\ttSch(a,b,i) 
&= \frac{1}{a}\binom{a}{i}\binom{a+b-1-i}{b-i}
= \frac{1}{b}\binom{b}{i}\binom{a+b-1-i}{a-i}.
\end{align*} 
\end{proposition}
\begin{proof}
The crucial observation is that the set of large $(a,b)$-Schr\"oder paths with $i$ diagonal steps can be generated by taking the set of $(a,b)$-Dyck paths with at least $i$ peaks, and replacing $i$ of the peaks with diagonal steps.  
Each large $(a,b)$-Schr\"oder path is obtained in a unique way in this construction, thus
\begin{align*}
\ttSch(a,b,i)
 &= \sum_{p\geq i} \binom{p}{i} \ttNar(a,b,p)
 = \sum_{p\geq i} \binom{p}{i} \frac{1}{a}\binom{a}{p}\binom{b-1}{b-p}\\
 &= \frac{1}{a}\binom{a}{i}\sum_{p\geq i}\binom{a-i}{p-i} \binom{b-1}{b-p}
 = \frac{1}{a}\binom{a}{i}\binom{a+b-1-i}{b-i}.
\end{align*}
\end{proof}

Following directly from the bijection $f$ constructed in Theorem~\ref{smallhalfofbig}, we have the next result which relates $\ttSch(a,b,i)$ and $\ttsch(a,b,i)$.
\begin{corollary} For coprime positive integers $a,b$, and $i=0,\ldots, a$,
$$\ttSch(a,b,i) = \ttsch(a,b,i) + \ttsch(a,b,i-1),$$
with the understanding that $\ttsch(a,b,-1) =0$.
\qed
\end{corollary} 

From this Corollary, we can deduce an explicit formula for the numbers $\ttsch(a,b,i)$.
\begin{proposition} For coprime positive integers $a,b$, and $i=0,\ldots, a-1$,
$$\ttsch(a,b,i) 
=\frac{1}{a}\binom{b-1}{i}\binom{a+b-1-i}{b}
= \frac{1}{b} \binom{a-1}{i} \binom{a+b-1-i}{a}.$$
\end{proposition}
\begin{proof}
Induct on $i$. 
By definition, $\ttsch(a,b,0)=\ttSch(a,b,0)=\ttCat(a,b)$, and one can check via a direct computation that $\ttSch(a,b,i) - \ttsch(a,b,i-1) = \ttsch(a,b,i)$.
\end{proof}


\section{$\nu$-Schr\"oder Trees}\label{sec.SchroderTrees}

In this section we introduce $\nu$-Schr\"oder trees, which generalize the $\nu$-trees of Ceballos, Padrol and Sarmiento~\cite{CPS20}.  
They showed that the rotation lattice of $\nu$-trees is an alternative description of the $\nu$-Tamari lattice, which is the $1$-skeleton of the $\nu$-associahedron.
Using the structural insight gained from the $\nu$-tree perspective, they showed that the $\nu$-Tamari lattice is isomorphic to the increasing flip poset of a suitably chosen subword complex, and solve a special case of Rubey's Lattice Conjecture. 

We define poset structures on $\nu$-Schr\"oder trees and $\nu$-Schr\"oder paths, and show that these posets are isomorphic.  In Section~\ref{sec.nuAssociahedron}, we show that these posets are an alternative description for the face poset of the $\nu$-associahedron.

Let $\nu$ be a lattice path from $(0,0)$ to $(b,a)$.
Let $R_\nu$ denote the region of the plane which lies weakly above $\nu$ inside the rectangle defined by $(0,0)$ and $(b,a)$.  
In Figure~\ref{fig:nuTrees}, $R_\nu$ is represented by the unshaded region in the rectangular grid.
Two lattice points $p$ and $q$ in $R_\nu$ are {\em $\nu$-incompatible} if and only if $p$ is southwest or northeast of $q$, and the smallest rectangle containing $p$ and $q$ is contained in $R_\nu$. 
We say that $p$ and $q$ are {\em $\nu$-compatible} if they are not $\nu$-incompatible.  

\begin{definition}\label{defn.nutrees}
A {\em $\nu$-Schr\"oder tree} is a set of $\nu$-compatible points in $R_\nu$ which includes the point $(0,b)$, such that each row and each column contains at least one point. 
The point $(0,b)$ in a $\nu$-Schr\"oder tree is the {\em root}, and the other points will be called {\em nodes}. 
A maximal collection of pairwise $\nu$-compatible lattice points in $R_\nu$ will be referred to as a {\em $\nu$-binary tree}.
Let $\calT_\nu$ denote the set of $\nu$-Schr\"oder trees.
\end{definition}

Note that $\nu$-binary trees are equal to the $\nu$-trees of \cite{CPS20}. We use the term $\nu$-binary tree to emphasize their binary nature. This way, the $\nu$-Schr\"oder trees generalize $\nu$-binary trees just as Schr\"oder trees generalize binary trees in the classical sense. 

It may seem peculiar that a collection of points is called a `tree', but this is justified as we may associate a non-crossing plane tree embedded in $R_\nu$ to each $\nu$-Schr\"oder tree $T$ as follows.
If a non-root node $p$ of $T$ in $R_\nu$ has a node above it in the same column or a node to the left of it in the same row, we connect them by an edge. 
Note that the $\nu$-compatibility of the nodes guarantees that it does not have both. 
However, it could have neither, in which case we consider the smallest rectangular box containing $p$ and exactly one other point $q$ of $T$. The point $q$ must be the northwest corner of such a box. 
The root guarantees the existence of such a box, and uniqueness follows from the fact that the northwest corners of two such hypothetical boxes would be $\nu$-incompatible. 
We then connect $p$ and $q$ by an edge. The resulting tree is guaranteed to be non-crossing, as otherwise the parent nodes of the two crossing edges would be $\nu$-incompatible. 

\begin{example}
Letting $\nu = \nu(3,5)$, Figure~\ref{fig:nuTrees} provides two examples of $\nu$-Schr\"oder trees. The region $R_\nu$ is the unshaded region weakly above $\nu$. The root is the node at $(0,3)$. Note that although the node $(2,3)$ is northeast of the node at $(0,1)$ in the left tree, they are $\nu$-compatible since the rectangle determined by them is not contained in $R_\nu$. No more nodes can be added to the left tree in Figure~\ref{fig:nuTrees} without introducing a pair of $\nu$-incompatible nodes, hence it is a $\nu$-binary tree. 
\begin{figure}[ht!]
\begin{tikzpicture}[scale=0.5]
\begin{scope}
    \draw[help lines] (0,0) grid (5,3);
    \draw[fill=gray, opacity=0.25] (0,0) rectangle (5,1); 
    \draw[fill=gray, opacity=0.25] (1,1) rectangle (5,2);
    \draw[fill=gray, opacity=0.25] (3,2) rectangle (5,3);

    \draw[ultra thick, RawSienna] (0,0) -- (0,1) -- (1,1);   
    \draw[ultra thick, RawSienna] (0,1) -- (0,3) -- (5,3);    
    \draw[ultra thick, RawSienna] (2,3) -- (2,2) -- (3,2);
    
    \vertex[draw=RawSienna, fill=LimeGreen] at (0,3) {};
    \vertex[draw=RawSienna, fill=LimeGreen] at (0,0) {};
    \vertex[draw=RawSienna, fill=LimeGreen] at (0,1) {};
    \vertex[draw=RawSienna, fill=LimeGreen] at (1,1) {};
    \vertex[draw=RawSienna, fill=LimeGreen] at (2,3) {};
    \vertex[draw=RawSienna, fill=LimeGreen] at (2,2) {};
    \vertex[draw=RawSienna, fill=LimeGreen] at (4,3) {};
    \vertex[draw=RawSienna, fill=LimeGreen] at (5,3) {};
    \vertex[draw=RawSienna, fill=LimeGreen] at (3,2) {};
\end{scope}
\begin{scope}[xshift=250]
    \draw[help lines] (0,0) grid (5,3);
    \draw[fill=gray, opacity=0.25] (0,0) rectangle (5,1); 
    \draw[fill=gray, opacity=0.25] (1,1) rectangle (5,2);
    \draw[fill=gray, opacity=0.25] (3,2) rectangle (5,3);

    \draw[ultra thick, RawSienna] (0,0) -- (0,3);
    \draw[ultra thick, RawSienna] (1,1) -- (0,3);    
    \draw[ultra thick, RawSienna] (0,3) -- (5,3);    
    \draw[ultra thick, RawSienna] (2,3) -- (3,2);
    
	\vertex[draw=RawSienna, fill=LimeGreen] at (0,3) {};
    \vertex[draw=RawSienna, fill=LimeGreen] at (0,0) {};
    \vertex[draw=RawSienna, fill=LimeGreen] at (1,1) {};
    \vertex[draw=RawSienna, fill=LimeGreen] at (2,3) {};
    \vertex[draw=RawSienna, fill=LimeGreen] at (4,3) {};
    \vertex[draw=RawSienna, fill=LimeGreen] at (5,3) {};
    \vertex[draw=RawSienna, fill=LimeGreen] at (3,2) {};
\end{scope}
\end{tikzpicture}
\caption{A $\nu$-binary tree (left) and a $\nu$-Schr\"oder tree (right), where $\nu=\nu(3,5)$. }
\label{fig:nuTrees}
\end{figure}
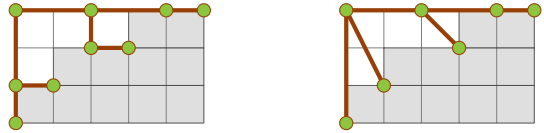
\label{ex:nuTrees}
\end{example}

\begin{definition} Let $p$, $q$ and $r$ be nodes in a $\nu$-Schr\"oder tree $S$ such that either $p$ is the first node above $q$ and $r$ is the first node to the right of $q$, or $p$ is the first node to the left of $q$ and $r$ is the first node below $q$. We define a \textit{contraction of $S$ at node $q$} as the $\nu$-Schr\"oder tree resulting from removing the node $q$ from $S$. When $p$ is above $q$ we call it a \textit{right contraction}, when $q$ is above $r$ we call it a \textit{left contraction}. Define a \textit{rotation at $q$} by removing the point $q$ and placing it in the other corner of the box determined by $p$ and $r$. If $p$ is above $q$, we call it a \textit{right rotation}, and if $q$ is above $r$, we call it a \textit{left rotation}. There is a third contraction possible, namely when $r$ is southeast of a non-leaf node $p$, with neither corner of the box determined by $p$ and $r$ containing a node of $S$. If removing the node $p$ yields a $\nu$-Schr\"oder tree, the removal of $p$ will be called a \textit{diagonal contraction}.  
\end{definition} 
Figures~\ref{LRcontraction} and $\ref{diagContraction}$ give diagrammatic illustrations of these definitions.

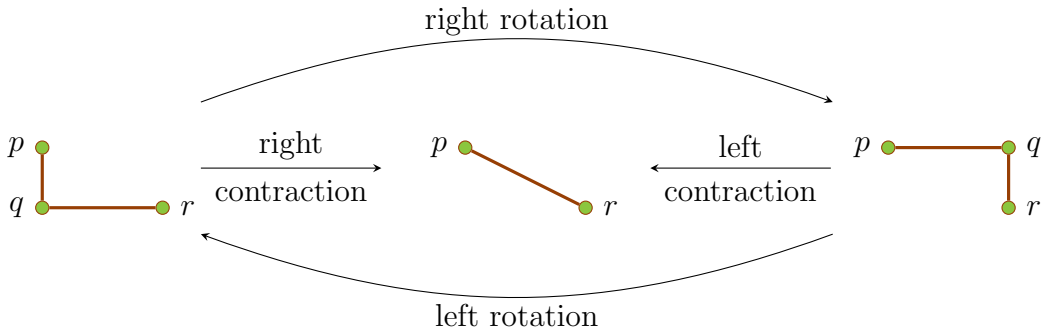
\begin{figure}[ht!]
\begin{tikzpicture}
\begin{scope}[xshift = 0, yshift = 0, scale=0.4]
	\vertex[draw=RawSienna, fill=LimeGreen, label=right:{$r$}] (r) at (4,0) {};
	\vertex[draw=RawSienna, fill=LimeGreen, label=left:{$q$}] (q)  at (0,0)  {};
	\vertex[draw=RawSienna, fill=LimeGreen, label=left:{$p$}] (p)  at (0,2)  {};	
	
	\draw[very thick, color=RawSienna] (p) to (q);
	\draw[very thick, color=RawSienna] (r) to (q);
\end{scope}
\begin{scope}[xshift=160, yshift=0, scale=0.4]
	\vertex[draw=RawSienna, fill=LimeGreen, label=right:{$r$}] (x)  at (4,0)  {};
	\vertex[draw=RawSienna, fill=LimeGreen, label=left:{$p$}] (z)  at (0,2)  {};	
	
	\draw[very thick, color=RawSienna] (x) to (z);
\end{scope}
\begin{scope}[xshift=320, yshift=0, scale=0.4]
	\vertex[draw=RawSienna, fill=LimeGreen, label=right:{$r$}] (x)  at (4,0)  {};
	\vertex[draw=RawSienna, fill=LimeGreen, label=right:{$q$}] (y)  at (4,2)  {};
	\vertex[draw=RawSienna, fill=LimeGreen, label=left:{$p$}] (z)  at (0,2)  {};	
	
	\draw[very thick, color=RawSienna] (x) to (y);
	\draw[very thick, color=RawSienna] (z) to (y);
\end{scope}
\begin{scope}[xshift=60, yshift=15, scale=0.6]
	\draw[-stealth] (0,0) to (4,0);
	\node[] (1) at (2,0.5) {right};
	\node[] (2) at (2,-0.5) {contraction};	
\end{scope}
\begin{scope}[xshift=230, yshift=15, scale=0.6]
	\draw[stealth-] (0,0) to (4,0);
	\node[] (1) at (2,0.5) {left};
	\node[] (2) at (2,-0.5) {contraction};
\end{scope}
\begin{scope}[xshift=60, yshift=40, scale=0.4]
    \draw[stealth-] (21,0) to[bend right=20] (0,0);
	\node[] (1) at (10.5,2.7) {right rotation};
\end{scope}
\begin{scope}[xshift=60, yshift=-10, scale=0.4]
    \draw[stealth-] (0,0) to[bend right=20] (21,0);
	\node[] (1) at (10.5,-2.7) {left rotation};    
\end{scope}
\end{tikzpicture}
\caption{A right and left contraction as intermediate steps in a right and left rotation, respectively.}
\label{LRcontraction}
\end{figure}

\begin{figure}[ht!]
\begin{tikzpicture}[scale=0.5]
\begin{scope}
    \draw[help lines] (0,0) grid (5,3);
    \draw[fill=gray, opacity=0.25] (0,0) rectangle (5,1); 
    \draw[fill=gray, opacity=0.25] (1,1) rectangle (5,2);
    \draw[fill=gray, opacity=0.25] (3,2) rectangle (5,3);

    \draw[ultra thick, RawSienna] (0,0) -- (0,3) -- (5,3);   
    \draw[ultra thick, RawSienna] (0,3) -- (1,2) -- (1,1);   
    \draw[ultra thick, RawSienna] (1,2) -- (3,2);
    
    \vertex[draw=RawSienna, fill=LimeGreen] at (0,3) {};
	\vertex[draw=RawSienna, fill=LimeGreen] at (0,3) {};		
	\vertex[draw=RawSienna, fill=LimeGreen] at (0,0) {};
    \vertex[draw=RawSienna, fill=LimeGreen] at (1,1) {};
	\vertex[draw=RawSienna, fill=LimeGreen] at (1,2) {};
    \vertex[draw=RawSienna, fill=LimeGreen] at (2,2) {};
	\vertex[draw=RawSienna, fill=LimeGreen] at (4,3) {};
	\vertex[draw=RawSienna, fill=LimeGreen] at (5,3) {};	
	\vertex[draw=RawSienna, fill=LimeGreen] at (3,2) {};
	
    \node at (3,-1) {};
\end{scope}
\begin{scope}[xshift=200,yshift=45]
    \draw[stealth-] (5,0) to[bend right=0] (0,0);
	\node[] (1) at (2.5,0.6) {diagonal};
    \node[] (1) at (2.5,-0.4) {contraction};
\end{scope}
\begin{scope}[xshift=400]
    \draw[help lines] (0,0) grid (5,3);
    \draw[fill=gray, opacity=0.25] (0,0) rectangle (5,1); 
    \draw[fill=gray, opacity=0.25] (1,1) rectangle (5,2);
    \draw[fill=gray, opacity=0.25] (3,2) rectangle (5,3);

    \draw[ultra thick, RawSienna] (0,0) -- (0,3) -- (5,3);   
    \draw[ultra thick, RawSienna] (0,3) -- (2,2) -- (3,2);   
    \draw[ultra thick, RawSienna] (0,3) -- (1,1);
    
	\vertex[draw=RawSienna, fill=LimeGreen] at (0,3) {};		
	\vertex[draw=RawSienna, fill=LimeGreen] at (0,0) {};
    \vertex[draw=RawSienna, fill=LimeGreen] at (1,1) {};
    \vertex[draw=RawSienna, fill=LimeGreen] at (2,2) {};
	\vertex[draw=RawSienna, fill=LimeGreen] at (4,3) {};
	\vertex[draw=RawSienna, fill=LimeGreen] at (5,3) {};	
	\vertex[draw=RawSienna, fill=LimeGreen] at (3,2) {};
\end{scope}
\end{tikzpicture}
\caption{A diagonal contraction at the node (1,2).}
\label{diagContraction}
\end{figure}
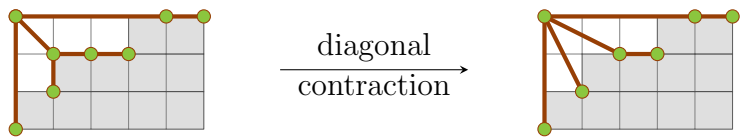

\begin{remark}
The term contraction comes from noticing that removing the node $q$ is equivalent to contracting the edge between $q$ and its neighbor closest to the root. 
The tree on the right in Figure
~\ref{fig:nuTrees} is formed from the tree on the left by contracting at the points $(0,1)$ and $(2,2)$. 
Performing a contraction on a $\nu$-binary tree $T$ can be thought of as an intermediate step in a left or right rotation of $\nu$-binary trees as defined in~\cite{CPS20}. 
See Figure~\ref{LRcontraction}. 
\end{remark}

\begin{proposition}
The set of $\nu$-Schr\"oder trees is the set of trees obtained from contracting $\nu$-binary trees. 
\end{proposition}

\begin{proof}
Since contraction always leaves at least one node in every row and column, performing a sequence of contractions on a $\nu$-binary tree results in a $\nu$-Schr\"oder tree. Conversely, given a $\nu$-Schr\"oder tree $T$, it is contained in a maximal set of $\nu$-compatible nodes, that is, a $\nu$-binary tree $T'$. Contracting $T'$ at the nodes not appearing in $T$ in any order yields $T$. 
\end{proof}

\begin{remark}
Since the set of $\nu$-binary trees determine a set of binary trees with labels {\em left} and {\em right} \cite[Lemma 2.4]{CPS20}, we can define the set of $\nu$-Schr\"oder trees as the set of labeled trees resulting from contracting internal edges in the corresponding set of binary trees. When contracting at a node $p$ labeled {\em left} or {\em right}, assign the label {\em middle} to all children with a different label than $p$. If $p$ has label {\em middle}, assign the label {\em middle} to all of its children.
Relabel the left and right children $E$ and $N$ respectively. In a contraction at $q$, each child of $q$ receives the label $D$. 
\end{remark} 

Next, we show that the leaves of a  $\nu$-Schr\"oder tree determine the path $\nu$, and vice versa. As a result, the path $\nu$ can be read from any $\nu$-Schr\"oder tree. 

\begin{lemma}[{\cite[Lemma 2.2]{CPS20}}] \label{WNlemma}
A non-root node in a $\nu$-binary tree has a node above it in the same column or to its left in the same row. 
\end{lemma}

\begin{proposition}
A node in a $\nu$-Schr\"oder tree is a leaf if and only if it is the starting point of a vertical run or an end point of a horizontal run in $\nu$. 
\end{proposition}

\begin{proof}
Note that contraction does not change the number of leaves, thus by Lemma~\ref{WNlemma} it suffices to only consider $\nu$-binary trees. 
If a node in a $\nu$-binary tree $T$ occurs at the starting point of a vertical run of $\nu$ or at the end point of a horizontal run, then it must be a leaf as it cannot have any nodes to its south or east. 

Conversely let $p$ be a leaf in $T$. Suppose toward a contradiction that $p$ is not the starting point of a vertical run or the end point of a horizontal run in $\nu$. 
Thus there is a lattice point $s \in R_\nu$ either below or to the right of $p$. Assume without loss of generality that $s$ is to the right of $p$. 
Then since $T$ is maximal, $s$ is not a node of $T$, and so it is $\nu$-incompatible with some $q \in T$. 
Since $q$ is $\nu$-compatible with $p$, it must be south or southeast of $p$. 
If it is south of $p$, then $s$ and $p$ are connected by an edge, and $p$ is not a leaf. 
If $s$ is southeast of $p$, then by Lemma~\ref{WNlemma} there is a node $t\in T$ that is either west of $s$ or north of $s$. 
By $\nu$-compatibility with $p$, the node $t$ cannot be southwest or northeast of $p$, thus $t$ is on the boundary of the rectangular box determined by $p$ and $s$. 
The nodes $s$ and $t$ are connected by an edge in $T$. Iterating the argument using $p$ and the point inside the box generates a path from $s$ to $p$, hence $p$ is not a leaf. 
It follows that $p$ is the starting point of a vertical run or the end point of a horizontal run in $\nu$.
\end{proof}

Since a $\nu$-Dyck path is determined by its horizontal and vertical runs, we have the following corollary. 

\begin{corollary}
The path $\nu$ is determined by a $\nu$-Schr\"oder tree. 
\qed
\end{corollary}

\begin{remark}
When $\nu = (NE)^n$ we recover the classical Schr\"oder trees, that is, trees with $n+1$ leaves where each non-leaf node has at least two children.
\end{remark}

\subsection{The bijection between $\nu$-Schr\"oder trees and $\nu$-Schr\"oder paths.}
\label{subsec.bijection}
The bijection $\varphi:\calT_\nu\rightarrow \calP_\nu$ given here between $\nu$-Schr\"oder trees and $\nu$-Schr\"oder paths is a generalization of the bijection between $\nu$-binary trees and $\nu$-Dyck paths given by Ceballos, Padrol and Sarmiento~\cite[Theorem 3.3]{CPS20}.  
Given a $\nu$-Schr\"oder tree $T$, we assign labels $N$, $E$ and $D$ to its non-root nodes as follows: if its parent node is in the same column then label it $N$, if its parent node is in the same row then label it $E$, and if its parent node is in neither then label it $D$. 

First define a \textit{right-flushing} map $\mathcal{R}$, which takes a $\nu$-Schr\"oder path $\mu$ and maps it to a $\nu$-Schr\"oder tree $T = \mathcal{R}(\mu)$ by right-flushing the lattice points of $\mu$ as follows. 
Begin by labeling the points in $\mu$ in the order they appear on the path, as it is traversed from the origin to $(b,a)$.
Starting from the bottom row in $R_\nu$ and proceeding upward, place the points in the same row of $R_\nu$ from right to left as far right as possible, while avoiding $x$-coordinates forbidden by previously right-flushed rows. 
An $x$-coordinate is {\em forbidden} if it corresponds to the initial point of an $E$ or $D$ step in $\mu$. 
We claim that the lattice points obtained by right-flushing all the lattice points in $\mu$ are the nodes of a $\nu$-Schr\"oder tree $T$.  
See the top of Figure~\ref{flushingMaps} for an example of the right-flushing map $\mathcal{R}$. 

We first check that $\mathcal{R}$ is well-defined. It is not immediately clear that right-flushing is always possible on a row, that is, that there is always an $x$-coordinate available in a row for the placement of a node. 
To verify that placing a node is always possible, suppose that we are right-flushing a point $p$ in the $\nu$-Schr\"oder path $\mu$. 
Let $\overline{p}$ denote the node to which $p$ is right-flushed. 
We need the number of lattice points in the row in $R_\nu$ on which $\overline{p}$ lies to be greater than the number of forbidden $x$-coordinates before $\overline{p}$. 
The latter is equal to the number of $E$ and $D$ steps before $p$. 
Let $\mathrm{horiz}_\nu(p)$ denote the maximal number of east steps that can be placed starting at $p$ before crossing $\nu$ (while remaining in the smallest rectangle containing $\nu$). For example, in Figure~\ref{flushingMaps}, $\mathrm{horiz}_\nu(4) = 2$ and $\mathrm{horiz}_\nu(9) = 3$.
The difference between the number of lattice points in the row with $p$ and the number of $E$ and $D$ steps before $p$ is equal to $\mathrm{horiz}_\nu(p) + 1$, and since this quantity is greater than or equal to one, there is a free column for the placement of $\overline{p}$. 

Next, we verify that $T=\mathcal{R}(\mu)$ is in fact a $\nu$-Schr\"oder tree. 
The construction guarantees the $\nu$-compatibility of the nodes, so it remains to verify the existence of the root, and that every row and column has a node. 
It is clear that every row has a node, as there is a lattice point of $\mu$ in every row. 
The total number of forbidden $x$-coordinates is the number of $E$ and $D$ steps in $\mu$, which is $b$, thus when flushing the last point of $\mu$, we must have $b$ forbidden $x$-coordinates, or in other words, nodes in $b$ columns. 
Note that the first column cannot be forbidden by any previous node, as such a forbidding node would correspond to a $E$ or $D$ step crossing $\nu$. 
Thus the last node must be placed in $(0,a)$, and so we have a node in each column, and a root.  

Now that $\mathcal{R}$ is well-defined, we define its inverse known as the {\em left-flushing} map $\mathcal{L}$, which left-flushes the nodes in a $\nu$-Schr\"oder tree $T$ to form a $\nu$-Schr\"oder path $\mu = \mathcal{L}(T)$ as follows.
First order the nodes in $T$ from bottom to top and right to left. 
Starting from the bottom row in $R_\nu$ and proceeding upward, place the nodes from left to right in the same row as far left as possible, while avoiding $x$-coordinates forbidden by previously left-flushed rows. 
The forbidden $x$-coordinates of a row are the $x$-coordinates of lattice points corresponding to nodes labeled $E$ or $D$ in $T$. 
We claim that the resulting collection of lattice points is a $\nu$-Schr\"oder path $\mu$. 
Note that by construction $\mu$ is the same $\nu$-Schr\"oder path as the one obtained by reading the labels in a post-order traversal of $T$. See the bottom of Figure~\ref{flushingMaps} for an example of the left-flushing map $\mathcal{L}$.

We verify that $\mathcal{L}$ is well-defined. First we check that left-flushing a node $\overline{p}$ in a $\nu$-tree is always possible, that is, that there is always an available lattice point of $R_\nu$ in the row of $\overline{p}$ in which to place $p$. 
We need more lattice points of $R_\nu$ on the row of $\overline{p}$ than the number of $x$-coordinates forbidden prior to $\overline{p}$. 
Let $\mathrm{hroot}_\nu(\overline{p})$ denote the number of nodes labeled $E$ or $D$ in the unique path from $\overline{p}$ to the root.
The difference between the number of lattice points of $R_\nu$ on the row of $\overline{p}$ and the number of $x$-coordinates forbidden prior to $\overline{p}$ is $\mathrm{hroot}_\nu(\overline{p}) + 1$. Since this quantity is greater than or equal to one, there is an available $x$-coordinate in the row of $\overline{p}$ in which to place $p$. 

It remains to check that $\mu=\mathcal{L}(T)$ is a $\nu$-Schr\"oder path. 
It is clear from the construction that $\mu$ is a lattice path with $N$, $E$ and $D$ steps. 
For any $\overline{p} \in T$ the quantity $\mathrm{hroot}_\nu(\overline{p})$ is one less than the difference between the number of lattice points of $R_\nu$ on the row of $\overline{p}$ and the number of $E$ and $D$ nodes read before $\overline{p}$, which is precisely $\mathrm{horiz}_\nu(p)$. 
Thus we have $\mathrm{hroot}_\nu(\overline{p}) = \mathrm{horiz}_\nu(p) \geq 0$ for any $p$, that is, $\mu$ lies weakly above $\nu$. 

\begin{figure}[ht!]
\begin{tikzpicture}
\begin{scope}[xshift=250, scale=0.5]
    \draw[help lines] (0,0) grid (7,6);
    \draw[fill=gray, opacity=0.25] (0,0) rectangle (7,1); 
    \draw[fill=gray, opacity=0.25] (1,1) rectangle (7,2);
    \draw[fill=gray, opacity=0.25] (3,2) rectangle (7,3);
    \draw[fill=gray, opacity=0.25] (5,3) rectangle (7,4);
    \draw[fill=gray, opacity=0.25] (5,4) rectangle (7,5);
    \draw[fill=gray, opacity=0.25] (6,5) rectangle (7,6);

    \draw [line join=round, decorate, decoration={
    zigzag, segment length=4,
    amplitude=.9,post=lineto,
    post length=2pt}]  (2,2) -- (2,6);
    \draw [line join=round, decorate, decoration={
    zigzag, segment length=4,
    amplitude=.9,post=lineto,
    post length=2pt}]  (3,2) -- (3,6);
    \draw [line join=round, decorate, decoration={
    zigzag, segment length=4,
    amplitude=.9,post=lineto,
    post length=2pt}]  (5,4) -- (5,6); 
    \draw [line join=round, decorate, decoration={
    zigzag, segment length=4,
    amplitude=.9,post=lineto,
    post length=2pt}]  (6,5) -- (6,6);    
    
    \draw[ultra thick, RawSienna] (0,0) -- (0,6) -- (7,6);
    \draw[ultra thick, RawSienna] (1,1) -- (1,6);    
    \draw[ultra thick, RawSienna] (4,6) -- (4,5) -- (6,5);
    \draw[ultra thick, RawSienna] (4,5) -- (5,4) -- (5,3);
    \draw[ultra thick, RawSienna] (1,6) -- (2,2) -- (3,2);
	
    \vertex[draw=RawSienna, fill=LimeGreen, label=above:\tiny{$\overline{12}$}] at (0,6) {};
	\vertex[draw=RawSienna, fill=LimeGreen, label=above:\tiny{$\overline{11}$}] at (1,6) {};
    \vertex[draw=RawSienna, fill=LimeGreen, label=below:\tiny{$\overline{1}$}] at (0,0) {};
	\vertex[draw=RawSienna, fill=LimeGreen, label=below:\tiny{$\overline{2}$}] at (1,1) {};
	\vertex[draw=RawSienna, fill=LimeGreen, label=below:\tiny{$\overline{4}$}] at (2,2) {};
    \vertex[draw=RawSienna, fill=LimeGreen, label=below:\tiny{$\overline{3}$}] at (3,2) {};
	\vertex[draw=RawSienna, fill=LimeGreen, label=above:\tiny{$\overline{10}$}] at (4,6) {};
    \vertex[draw=RawSienna, fill=LimeGreen, label=below:\tiny{$\overline{8}$}] at (4,5) {};
    \vertex[draw=RawSienna, fill=LimeGreen, label=right:\tiny{$\overline{6}$}] at (5,4) {};
    \vertex[draw=RawSienna, fill=LimeGreen, label=right:\tiny{$\overline{5}$}] at (5,3) {};
    \vertex[draw=RawSienna, fill=LimeGreen, label=right:\tiny{$\overline{7}$}] at (6,5) {};
    \vertex[draw=RawSienna, fill=LimeGreen, label=above:\tiny{$\overline{9}$}] at (7,6) {};  
\end{scope}
\begin{scope}[xshift=0, yshift=-140, scale=0.5]
    \draw[help lines] (0,0) grid (7,6);
    \draw[fill=gray, opacity=0.25] (0,0) rectangle (7,1); 
    \draw[fill=gray, opacity=0.25] (1,1) rectangle (7,2);
    \draw[fill=gray, opacity=0.25] (3,2) rectangle (7,3);
    \draw[fill=gray, opacity=0.25] (5,3) rectangle (7,4);
    \draw[fill=gray, opacity=0.25] (5,4) rectangle (7,5);
    \draw[fill=gray, opacity=0.25] (6,5) rectangle (7,6);

    \draw[ultra thick, RawSienna] (0,0) -- (0,6) -- (7,6);
    \draw[ultra thick, RawSienna] (1,1) -- (1,6);    
    \draw[ultra thick, RawSienna] (4,6) -- (4,5) -- (6,5);
    \draw[ultra thick, RawSienna] (4,5) -- (5,4) -- (5,3);
    \draw[ultra thick, RawSienna] (1,6) -- (2,2) -- (3,2);
    
    \vertex[draw=RawSienna, fill=LimeGreen] at (0,6) {};
	\vertex[draw=RawSienna, fill=LimeGreen] at (1,6) {};	
    \vertex[draw=RawSienna, fill=LimeGreen] at (0,0) {};
	\vertex[draw=RawSienna, fill=LimeGreen] at (1,1) {};
	\vertex[draw=RawSienna, fill=LimeGreen] at (2,2) {};
    \vertex[draw=RawSienna, fill=LimeGreen] at (3,2) {};
	\vertex[draw=RawSienna, fill=LimeGreen] at (4,6) {};
    \vertex[draw=RawSienna, fill=LimeGreen] at (4,5) {};	
    \vertex[draw=RawSienna, fill=LimeGreen] at (5,4) {};
    \vertex[draw=RawSienna, fill=LimeGreen] at (5,3) {};
    \vertex[draw=RawSienna, fill=LimeGreen] at (6,5) {};
    \vertex[draw=RawSienna, fill=LimeGreen] at (7,6) {};    

	\node[] at (-0.6,0) {\tiny{$N$}};
	\node[] at (0.6,1.5) {\tiny{$N$}};
	\node[] at (2.3,2.5) {\tiny{$D$}};
	\node[] at (3.3,2.5) {\tiny{$E$}};	
	\node[] at (4.6,2.5) {\tiny{$N$}};		
	\node[] at (4.6,3.5) {\tiny{$D$}};
	\node[] at (3.6,4.6) {\tiny{$N$}};
	\node[] at (1.2,6.6) {\tiny{$E$}};	
	\node[] at (5.6,5.5) {\tiny{$E$}};
	\node[] at (4.2,6.6) {\tiny{$E$}};
	\node[] at (7.2,6.6) {\tiny{$E$}};			
\end{scope}
\begin{scope}[xshift=0, yshift=0, scale=0.5]
    \draw[help lines] (0,0) grid (7,6);
    \draw[fill=gray, opacity=0.25] (0,0) rectangle (7,1); 
    \draw[fill=gray, opacity=0.25] (1,1) rectangle (7,2);
    \draw[fill=gray, opacity=0.25] (3,2) rectangle (7,3);
    \draw[fill=gray, opacity=0.25] (5,3) rectangle (7,4);
    \draw[fill=gray, opacity=0.25] (5,4) rectangle (7,5);
    \draw[fill=gray, opacity=0.25] (6,5) rectangle (7,6);

    \draw[very thick, red] (0,0) -- (0,2) -- (1,2) -- (2,3) -- (2,4) -- (3,5) -- (4,5) -- (4,6) --  (7,6);

    \vertex[fill=black, minimum size=4pt, label=left:{\tiny$1$}] at (0,0) {};
    \vertex[fill=black, minimum size=4pt, label=left:{\tiny$2$}] at (0,1) {};
    \vertex[fill=black, minimum size=4pt, label=left:{\tiny$3$}] at (0,2) {};
    \vertex[fill=black, minimum size=4pt] at (1,2) {};
    \vertex[fill=black, minimum size=4pt] at (2,3) {};
    \vertex[fill=black, minimum size=4pt] at (2,4) {};
    \vertex[fill=black, minimum size=4pt] at (3,5) {};
    \vertex[fill=black, minimum size=4pt] at (4,5) {};
    \vertex[fill=black, minimum size=4pt, label=above:{\tiny$9$}] at (4,6) {};
    \vertex[fill=black, minimum size=4pt, label=above:{\tiny$10$}] at (5,6) {};
    \vertex[fill=black, minimum size=4pt, label=above:{\tiny$11$}] at (6,6) {};
    \vertex[fill=black, minimum size=4pt, label=above:{\tiny$12$}] at (7,6) {};
    
    \node[] at (0.6,2.4) {\tiny{$4$}};
    \node[] at (1.6,3.3) {\tiny{$5$}};
    \node[] at (1.6,4.3) {\tiny{$6$}};
    \node[] at (2.7,5.3) {\tiny{$7$}};
    \node[] at (3.7,5.4) {\tiny{$8$}};  
\end{scope}
\begin{scope}[xshift=250, yshift=-140, scale=0.5]
    \draw[help lines] (0,0) grid (7,6);
    \draw[fill=gray, opacity=0.25] (0,0) rectangle (7,1); 
    \draw[fill=gray, opacity=0.25] (1,1) rectangle (7,2);
    \draw[fill=gray, opacity=0.25] (3,2) rectangle (7,3);
    \draw[fill=gray, opacity=0.25] (5,3) rectangle (7,4);
    \draw[fill=gray, opacity=0.25] (5,4) rectangle (7,5);
    \draw[fill=gray, opacity=0.25] (6,5) rectangle (7,6);

    \draw[very thick, red] (0,0) -- (0,2) -- (1,2) -- (2,3) -- (2,4) -- (3,5) -- (4,5) -- (4,6) --  (7,6);

    \node[circle, fill=black, inner sep=1.5pt] at (0,0) {};
    \node[circle, fill=black, inner sep=1.5pt] at (0,1) {};
    \node[circle, fill=black, inner sep=1.5pt] at (0,2) {};
    \node[circle, fill=black, inner sep=1.5pt] at (1,2) {};
    \node[circle, fill=black, inner sep=1.5pt] at (2,3) {};
    \node[circle, fill=black, inner sep=1.5pt] at (2,4) {};
    \node[circle, fill=black, inner sep=1.5pt] at (3,5) {};
    \node[circle, fill=black, inner sep=1.5pt] at (4,5) {};
    \node[circle, fill=black, inner sep=1.5pt] at (4,6) {};
    \node[circle, fill=black, inner sep=1.5pt] at (5,6) {};
    \node[circle, fill=black, inner sep=1.5pt] at (6,6) {};
    \node[circle, fill=black, inner sep=1.5pt] at (7,6) {};

    \draw [line join=round, decorate, decoration={
    zigzag, segment length=4,
    amplitude=.9,post=lineto,
    post length=2pt}]  (0,2) -- (0,6);    
    \draw [line join=round, decorate, decoration={
    zigzag, segment length=4,
    amplitude=.9,post=lineto,
    post length=2pt}]  (1,2) -- (1,6);
    \draw [line join=round, decorate, decoration={
    zigzag, segment length=4,
    amplitude=.9,post=lineto,
    post length=2pt}]  (2,4) -- (2,6);
    \draw [line join=round, decorate, decoration={
    zigzag, segment length=4,
    amplitude=.9,post=lineto,
    post length=2pt}]  (3,5) -- (3,6);
\end{scope}

\begin{scope}[xshift=150, yshift=40, scale=0.5]
    \draw[-stealth] (0,0) to [bend left=0] (4,0);
    \node[] at (2,1) {$\mathcal{R}$};
\end{scope}
\begin{scope}[xshift=150, yshift=-100, scale=0.5]
    \draw[-stealth] (0,0) to [bend left=0] (4,0);
    \node[] at (2,1) {$\mathcal{L}$};
\end{scope}
\end{tikzpicture}
\caption{The right-flushing map $\mathcal{R}$ (top) and the left-flushing map $\mathcal{L}$ (bottom). The action of $\mathcal{L}$ is equivalent to reading the labels of the $\nu$-Schr\"oder tree in post-order traversal starting at the root and going counter-clockwise.  The zigzag lines indicate the forbidden $x$-coordinates.
}
\label{flushingMaps}
\end{figure}
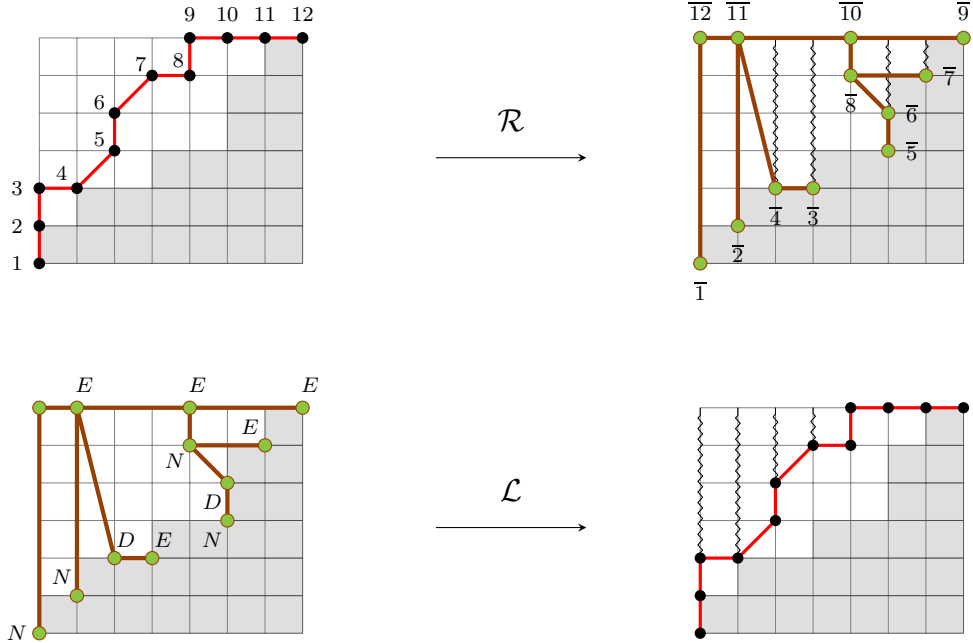

Finally, we check that the right and left flushing maps $\mathcal{R}$ and $\mathcal{L}$ are inverses. 
Any $\nu$-Schr\"oder path $\mu$ is uniquely determined by its lattice points. The $x$-coordinate of a point $p$ in $\mu$ is determined by the number of $E$ and $D$ steps before $p$, which is precisely the number of forbidden $x$-coordinates before $\overline{p}$ in $\mathcal{R}(\mu)$. Therefore the $x$-coordinate of $\mathcal{L}(\overline{p})$ is the same as that of $p$, and since $\mathcal{R}$ and $\mathcal{L}$ do not alter the $y$-coordinates, we have $\mathcal{L}(\mathcal{R}(p))=p$. Note that $\mathcal{R}$ is injective, as two different $\nu$-Schr\"oder paths have at least one row with a different number of lattice points, and so the corresponding $\nu$-Schr\"oder trees differ on that row. 

The next theorem now follows.

\begin{theorem}\label{treePathBijection}
The map $\varphi:\calT_\nu\rightarrow \calP_\nu$ is a bijection between the set of $\nu$-Schr\"oder trees and the set of $\nu$-Schr\"oder paths. \qed
\end{theorem}

\subsection{The posets of $\nu$-Schr\"oder trees and paths.}

The set of $\nu$-Schr\"oder trees satisfy a partial order induced by the covering relation $T \prec T'$ if and only if $T'$ is a contraction of $T$. 
We call this the \textit{poset of $\nu$-Schr\"oder trees}. 

To define a poset on $\nu$-Schr\"oder paths, we translate contractions of $\nu$-Schr\"oder trees to $\nu$-Schr\"oder paths. 
The right, left and diagonal contractions are considered separately, as they correspond to different contraction moves on $\nu$-Schr\"oder paths. 

Let $T$ be a $\nu$-Schr\"oder tree.
First we consider a right contraction of $T$ at a node $\overline{q}$ with parent node $\overline{p}$ above $\overline{q}$ and with a child node $\overline{r}$ to the right of $\overline{q}$. 
The labels of the nodes $\overline{q}$ and $\overline{r}$ are $N$ and $E$ respectively.  
Contracting at $\overline{q}$ removes the node $\overline{q}$ and the label on the node $\overline{r}$ becomes $D$. 
This corresponds to replacing an $E$ step and a $N$ step in $\varphi(T)$ with a $D$ step. 
In the counterclockwise post-order traversal of $T$, the $E$ and $N$ steps are consecutive, and so correspond to a valley in $\varphi(T)$. 
Thus a right contraction in $T$ corresponds to replacing a valley in $\varphi(T)$ with a $D$ step.   

Next, consider a left contraction in $T$ at a node $\overline{q}$ with parent node $\overline{p}$ to the left of $\overline{q}$ and with a child node $\overline{r}$ below $\overline{q}$. 
As in the case above, contracting at $\overline{q}$ replaces an $E$ step and $N$ step with a $D$ step at $\overline{r}$. However, this time the $N$ and $E$ steps are not necessarily consecutive in $\varphi(T)$, as $\overline{q}$ may have other children which are read before $\overline{q}$ in the post-order traversal of $T$. The node $\overline{r}$ is the previous node in the post order traversal of $T$ such that $\mathrm{hroot}_\nu(\overline{r}) = \mathrm{hroot}_\nu(\overline{q})$. 
Recall from Section~\ref{subsec.bijection} that $\mathrm{hroot}_\nu(\overline{x})=\mathrm{horiz}_\nu(x)$. Therefore, $r$ is the previous lattice point on $\varphi(T)$ such that $r$ is the initial point of an $N$ step and $\mathrm{horiz}_\nu(r) = \mathrm{horiz}_\nu(q)$. Left contraction deletes this pair of $E$ and $N$ steps, and places a $D$ step at $r$. See Figure~\ref{LRpathContraction} for an example. 

Lastly, consider a diagonal contraction in $T$ at a node $\overline{r}$ with parent node $\overline{p}$. 
Note that $\overline{r}$ must have a left child $\overline{s}$ and a right child $\overline{t}$, as otherwise contracting at $\overline{r}$ would not yield a $\nu$-Schr\"oder tree (either the row or column of $\overline{r}$ would not contain a node). 
The labels of the nodes $\overline{r}$, $\overline{s}$, and $\overline{t}$ are $D$, $N$, and $E$ respectively. 
Contracting at $\overline{r}$ changes the labels of both $\overline{s}$ and $\overline{t}$ to $D$. 
In the post-order traversal of the tree, this contraction corresponds to replacing the label $N$ at $\overline{s}$ with $D$, replacing the label $E$ at $\overline{t}$ with $D$, and removing the point $\overline{r}$ labeled $D$. Note that $\overline{s}$ is the first point before $\overline{t}$ in the post-order traversal satisfying $\mathrm{hroot}_\nu(\overline{s}) = \mathrm{hroot}_\nu(\overline{r})=\mathrm{hroot}_\nu(\overline{t}) -1$. 
Therefore, $s$ is the previous point on $\varphi(T)$ such that $\mathrm{horiz}_\nu(s) = \mathrm{horiz}_\nu(r) = \mathrm{horiz}_\nu(t)-1$. 
Diagonal contraction thus deletes the step $E$ with end point $r$ and the step $N$ with initial point $s$, and places a $D$ step at $s$. See Figure~\ref{diagContraction2} for an example.

\begin{figure}[ht!]
\begin{tikzpicture}[scale=1.1]
\begin{scope}[xshift = -10, yshift = 0, scale=0.5]
\draw[help lines] (0,0) grid (5,3);
    \draw[fill=gray, opacity=0.25] (0,0) rectangle (5,1); 
    \draw[fill=gray, opacity=0.25] (1,1) rectangle (5,2);
    \draw[fill=gray, opacity=0.25] (3,2) rectangle (5,3);

    \draw[ultra thick, RawSienna] (0,0) -- (0,1) -- (1,1);
    \draw[ultra thick, RawSienna] (0,1) -- (0,3) -- (5,3);
    \draw[ultra thick, RawSienna] (2,3) -- (2,2) -- (3,2);
    
	\vertex[draw=RawSienna, fill=LimeGreen] at (0,3) {};		
	\vertex[draw=RawSienna, fill=LimeGreen] at (0,0) {};
	\vertex[draw=RawSienna, fill=LimeGreen] at (0,1) {};	
	\vertex[draw=RawSienna, fill=LimeGreen] at (1,1) {};
	\vertex[draw=RawSienna, fill=LimeGreen] at (2,3) {};
	\vertex[draw=RawSienna, fill=LimeGreen] at (2,2) {};
	\vertex[draw=RawSienna, fill=LimeGreen] at (4,3) {};
	\vertex[draw=RawSienna, fill=LimeGreen] at (5,3) {};	
	\vertex[draw=RawSienna, fill=LimeGreen] at (3,2) {};	
	
	\node[] at (-0.6,0) {\tiny{$N$}};
	\node[] at (-0.6,1) {\tiny{$N$}};
	\node[] at (1.3,0.4) {\tiny{$E$}};	
	\node[] at (2,3.5) {\tiny{$E$}};
	\node[] at (4,3.5) {\tiny{$E$}};
	\node[] at (5,3.5) {\tiny{$E$}};
	\node[] at (2.4,1.5) {\tiny{$N$}};
	\node[] at (3.4,1.5) {\tiny{$E$}};	
\end{scope}
\begin{scope}[xshift=150, yshift=0, scale=0.5]
	\draw[help lines] (0,0) grid (5,3);
    \draw[fill=gray, opacity=0.25] (0,0) rectangle (5,1); 
    \draw[fill=gray, opacity=0.25] (1,1) rectangle (5,2);
    \draw[fill=gray, opacity=0.25] (3,2) rectangle (5,3);

    \draw[ultra thick, RawSienna] (0,3) -- (1,1);
    \draw[ultra thick, RawSienna] (0,0) -- (0,3) -- (5,3);
    \draw[ultra thick, RawSienna] (2,3) -- (2,2) -- (3,2);
    
	\vertex[draw=RawSienna, fill=LimeGreen] at (0,3) {};		
	\vertex[draw=RawSienna, fill=LimeGreen] at (0,0) {};
	\vertex[draw=RawSienna, fill=LimeGreen] at (1,1) {};
	\vertex[draw=RawSienna, fill=LimeGreen] at (2,3) {};
	\vertex[draw=RawSienna, fill=LimeGreen] at (2,2) {};
	\vertex[draw=RawSienna, fill=LimeGreen] at (4,3) {};
	\vertex[draw=RawSienna, fill=LimeGreen] at (5,3) {};	
	\vertex[draw=RawSienna, fill=LimeGreen] at (3,2) {};	
	
	\node[] at (-0.6,0) {\tiny{$N$}};
	\node[] at (1.3,0.4) {\tiny{$D$}};	
	\node[] at (2,3.5) {\tiny{$E$}};
	\node[] at (4,3.5) {\tiny{$E$}};
	\node[] at (5,3.5) {\tiny{$E$}};
	\node[] at (2.4,1.5) {\tiny{$N$}};
	\node[] at (3.4,1.5) {\tiny{$E$}};		
\end{scope}
\begin{scope}[xshift=300, yshift=0, scale=0.5]
	\draw[help lines] (0,0) grid (5,3);
    \draw[fill=gray, opacity=0.25] (0,0) rectangle (5,1); 
    \draw[fill=gray, opacity=0.25] (1,1) rectangle (5,2);
    \draw[fill=gray, opacity=0.25] (3,2) rectangle (5,3);

    \draw[ultra thick, RawSienna] (1,3) -- (1,1);
    \draw[ultra thick, RawSienna] (0,0) -- (0,3) -- (5,3);
    \draw[ultra thick, RawSienna] (2,3) -- (2,2) -- (3,2);
    
	\vertex[draw=RawSienna, fill=LimeGreen] at (0,3) {};		
	\vertex[draw=RawSienna, fill=LimeGreen] at (0,0) {};
	\vertex[draw=RawSienna, fill=LimeGreen] at (1,3) {};	
	\vertex[draw=RawSienna, fill=LimeGreen] at (1,1) {};
	\vertex[draw=RawSienna, fill=LimeGreen] at (2,3) {};
	\vertex[draw=RawSienna, fill=LimeGreen] at (2,2) {};
	\vertex[draw=RawSienna, fill=LimeGreen] at (4,3) {};
	\vertex[draw=RawSienna, fill=LimeGreen] at (5,3) {};	
	\vertex[draw=RawSienna, fill=LimeGreen] at (3,2) {};	
	
	\node[] at (-0.6,0) {\tiny{$N$}};
	\node[] at (1.3,0.4) {\tiny{$N$}};
	\node[] at (1,3.5) {\tiny{$E$}};
	\node[] at (2,3.5) {\tiny{$E$}};
	\node[] at (4,3.5) {\tiny{$E$}};
	\node[] at (5,3.5) {\tiny{$E$}};
	\node[] at (2.4,1.5) {\tiny{$N$}};
	\node[] at (3.4,1.5) {\tiny{$E$}};		
\end{scope}
\begin{scope}[xshift=75, yshift=20, scale=0.6]
	\draw[-stealth] (0,0) to (3,0);
	\node[] (1) at (1.5,0.5) {right};
	\node[] (2) at (1.5,-0.5) {contraction};	
\end{scope}
\begin{scope}[xshift=230, yshift=20, scale=0.6]
	\draw[stealth-] (0,0) to (3,0);
	\node[] (1) at (1.5,0.5) {left};
	\node[] (2) at (1.5,-0.5) {contraction};
\end{scope}
\begin{scope}[xshift = -10, yshift=-110, scale=0.5]
\draw[help lines] (0,0) grid (5,3);
    \draw[fill=gray, opacity=0.25] (0,0) rectangle (5,1); 
    \draw[fill=gray, opacity=0.25] (1,1) rectangle (5,2);
    \draw[fill=gray, opacity=0.25] (3,2) rectangle (5,3);

    \draw[very thick, red] (0,0) -- (0,1) -- (1,1) -- (1,2) -- (2,2) -- (2,3) -- (5,3);
\end{scope}
\begin{scope}[xshift=150, yshift=-110, scale=0.5]
\draw[help lines] (0,0) grid (5,3);
    \draw[fill=gray, opacity=0.25] (0,0) rectangle (5,1); 
    \draw[fill=gray, opacity=0.25] (1,1) rectangle (5,2);
    \draw[fill=gray, opacity=0.25] (3,2) rectangle (5,3);

    \draw[very thick, red] (0,0) -- (0,1) -- (1,2) -- (2,2) -- (2,3) -- (5,3);
\end{scope}
\begin{scope}[xshift=300, yshift=-110, scale=0.5]
\draw[help lines] (0,0) grid (5,3);
    \draw[fill=gray, opacity=0.25] (0,0) rectangle (5,1); 
    \draw[fill=gray, opacity=0.25] (1,1) rectangle (5,2);
    \draw[fill=gray, opacity=0.25] (3,2) rectangle (5,3);

    \draw[very thick, red] (0,0) -- (0,2) -- (1,2) -- (1,3) -- (5,3);
    
	\node[circle, draw = black, fill=black, inner sep=1pt, label=left:{\scriptsize$r$}] at (0,1) {};
	\node[circle, draw = black, fill=black, inner sep=1pt, label=above:{\scriptsize$q$}] at (4,3) {};
\end{scope}
\begin{scope}[xshift=75, yshift=-90, scale=0.6]
	\draw[-stealth] (0,0) to (3,0);
	\node[] (1) at (1.5,0.5) {right};
	\node[] (2) at (1.5,-0.5) {contraction};	
\end{scope}
\begin{scope}[xshift=230, yshift=-90, scale=0.6]
	\draw[stealth-] (0,0) to (3,0);
	\node[] (1) at (1.5,0.5) {left};
	\node[] (2) at (1.5,-0.5) {contraction};
\end{scope}
\begin{scope}[xshift=23, yshift=-20, scale=0.6]
	\draw[-stealth] (0,0) to (0,-2);
	\node[] (1) at (0.5,-1) {$\varphi$};
\end{scope}
\begin{scope}[xshift=183, yshift=-20, scale=0.6]
	\draw[-stealth] (0,0) to (0,-2);
	\node[] (1) at (0.5,-1) {$\varphi$};
\end{scope}
\begin{scope}[xshift=333, yshift=-20, scale=0.6]
	\draw[-stealth] (0,0) to (0,-2);
	\node[] (1) at (0.5,-1) {$\varphi$};
\end{scope}
\end{tikzpicture}
\caption{A right and left contraction of a pair of $(3,5)$-Schr\"oder trees, and the corresponding contractions in the associated $(3,5)$-Schr\"oder paths.}
\label{LRpathContraction}
\end{figure}
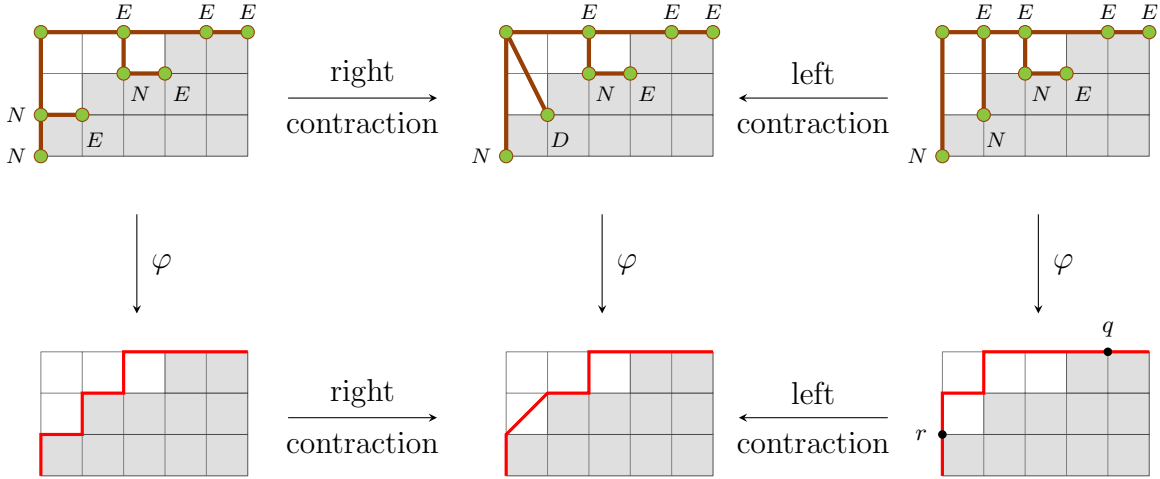

\begin{figure}[ht!]
\begin{tikzpicture}[scale=1.1]
\begin{scope}[scale=0.5]
    \draw[help lines] (0,0) grid (5,3);
    \draw[fill=gray, opacity=0.25] (0,0) rectangle (5,1); 
    \draw[fill=gray, opacity=0.25] (1,1) rectangle (5,2);
    \draw[fill=gray, opacity=0.25] (3,2) rectangle (5,3);

    \draw[ultra thick, RawSienna] (0,0) -- (0,3) -- (5,3);   
    \draw[ultra thick, RawSienna] (0,3) -- (1,2) -- (1,1);   
    \draw[ultra thick, RawSienna] (1,2) -- (3,2);
    
	\vertex[draw=RawSienna, fill=LimeGreen] at (0,3) {};		
	\vertex[draw=RawSienna, fill=LimeGreen] at (0,0) {};
    \vertex[draw=RawSienna, fill=LimeGreen] at (1,1) {};
	\vertex[draw=RawSienna, fill=LimeGreen] at (1,2) {};
    \vertex[draw=RawSienna, fill=LimeGreen] at (2,2) {};
	\vertex[draw=RawSienna, fill=LimeGreen] at (4,3) {};
	\vertex[draw=RawSienna, fill=LimeGreen] at (5,3) {};	
	\vertex[draw=RawSienna, fill=LimeGreen] at (3,2) {};
	
	\node[] at (-0.6,0) {\tiny{$N$}};
	\node[] at (4,3.5) {\tiny{$E$}};
	\node[] at (5,3.5) {\tiny{$E$}};
	\node[] at (2.4,1.5) {\tiny{$E$}};
	\node[] at (3.4,1.5) {\tiny{$E$}};	
	\node[] at (1.4,0.5) {\tiny{$N$}};	
	\node[] at (1.4,2.5) {\tiny{$D$}};	    
\end{scope}

\begin{scope}[xshift=200, scale=0.5]
    \draw[help lines] (0,0) grid (5,3);
    \draw[fill=gray, opacity=0.25] (0,0) rectangle (5,1); 
    \draw[fill=gray, opacity=0.25] (1,1) rectangle (5,2);
    \draw[fill=gray, opacity=0.25] (3,2) rectangle (5,3);

    \draw[ultra thick, RawSienna] (0,0) -- (0,3) -- (5,3);   
    \draw[ultra thick, RawSienna] (0,3) -- (2,2) -- (3,2);   
    \draw[ultra thick, RawSienna] (0,3) -- (1,1);
    
	\vertex[draw=RawSienna, fill=LimeGreen] at (0,3) {};		
	\vertex[draw=RawSienna, fill=LimeGreen] at (0,0) {};
    \vertex[draw=RawSienna, fill=LimeGreen] at (1,1) {};
    \vertex[draw=RawSienna, fill=LimeGreen] at (2,2) {};
	\vertex[draw=RawSienna, fill=LimeGreen] at (4,3) {};
	\vertex[draw=RawSienna, fill=LimeGreen] at (5,3) {};	
	\vertex[draw=RawSienna, fill=LimeGreen] at (3,2) {};
	
	\node[] at (-0.6,0) {\tiny{$N$}};
	\node[] at (4,3.5) {\tiny{$E$}};
	\node[] at (5,3.5) {\tiny{$E$}};
	\node[] at (2.4,1.5) {\tiny{$D$}};
	\node[] at (3.4,1.5) {\tiny{$E$}};	
	\node[] at (1.4,0.5) {\tiny{$D$}};	
\end{scope}

\begin{scope}[xshift=0, yshift=-110, scale=0.5]
    \draw[help lines] (0,0) grid (5,3);
    \draw[fill=gray, opacity=0.25] (0,0) rectangle (5,1); 
    \draw[fill=gray, opacity=0.25] (1,1) rectangle (5,2);
    \draw[fill=gray, opacity=0.25] (3,2) rectangle (5,3);

    \draw[very thick, red] (0,0) -- (0,2) -- (2,2) -- (3,3) -- (5,3);
    
	\node[circle, draw = black, fill=black, inner sep=1pt, label=left:{\scriptsize$s$}] at (0,1) {};
	\node[circle, draw = black, fill=black, inner sep=1pt, label=above:{\scriptsize$r$}] at (2,2) {};
\end{scope}

\begin{scope}[xshift=200, yshift=-110, scale=0.5]
    \draw[help lines] (0,0) grid (5,3);
    \draw[fill=gray, opacity=0.25] (0,0) rectangle (5,1); 
    \draw[fill=gray, opacity=0.25] (1,1) rectangle (5,2);
    \draw[fill=gray, opacity=0.25] (3,2) rectangle (5,3);

    \draw[very thick, red] (0,0) -- (0,1) -- (1,2) -- (2,2) -- (3,3) -- (5,3);
\end{scope}

\begin{scope}[xshift=95,yshift=20, scale=0.5]
    \draw[stealth-] (6,0) to[bend right=0] (0,0);
	\node[] (1) at (3,0.7) {diagonal};
    \node[] (1) at (3,-0.5) {contraction};
\end{scope}
\begin{scope}[xshift=95,yshift=-90, scale=0.5]
    \draw[stealth-] (6,0) to[bend right=0] (0,0);
	\node[] (1) at (3,0.7) {diagonal};
    \node[] (1) at (3,-0.5) {contraction};
\end{scope}
\begin{scope}[xshift=33, yshift=-15, scale=0.6]
	\draw[-stealth] (0,0) to (0,-2);
	\node[] (1) at (0.5,-1) {$\varphi$};
\end{scope}
\begin{scope}[xshift=234, yshift=-15, scale=0.6]
	\draw[-stealth] (0,0) to (0,-2);
	\node[] (1) at (0.5,-1) {$\varphi$};
\end{scope}
\end{tikzpicture}
\caption{A diagonal contraction of a $(3,5)$-Schr\"oder tree and the corresponding diagonal contraction in the associated $(3,5)$-Schr\"oder path.}
\label{diagContraction2}
\end{figure}
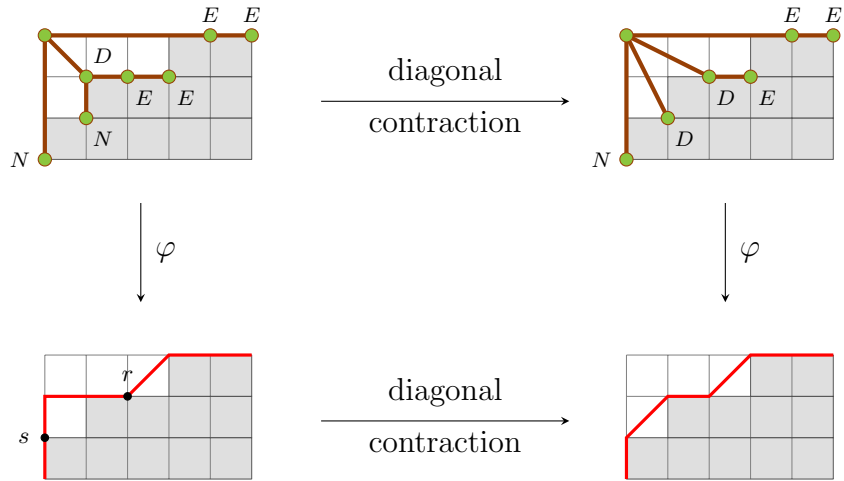

The set of $\nu$-Schr\"oder paths then form a poset with the cover relation inherited from the poset of $\nu$-Schr\"oder trees. 

\begin{definition}
The (contraction) poset $P_\nu$ of $\nu$-Schr\"oder paths is the set of $\nu$-Schr\"oder paths with cover relation $\mu \prec \lambda$ if and only if $\lambda$ is formed from $\mu$ by a contraction. The contraction moves are the following:
\begin{itemize}
    \item[1.] \textbf{Right Contraction:} Replace a consecutive $EN$ pair with $D$. 
    \item[2.] \textbf{Left Contraction:} Delete an $E$ step with initial point $q$, along with the preceding $N$ step with initial point $r$ satisfying $\mathrm{horiz}_\nu(r) = \mathrm{horiz}_\nu(q)$. Shift the subpath between the deleted steps one unit to the right, and place a $D$ step at $r$. 
    \item[3.] \textbf{Diagonal Contraction:} Delete an $E$ step ending at a point $r$, which is the initial point of a $D$ step, along with the preceding $N$ step with initial point $s$ satisfying $\mathrm{horiz}_\nu(s) = \mathrm{horiz}_\nu(r)$. Shift the subpath between the deleted steps one unit to the right, and place a $D$ step at $s$.
\end{itemize}
\end{definition}

See Figure~\ref{fig:53poset} for an example of the poset of $\nu$-Schr\"oder paths for the rational $\nu=\nu(3,5)$.

By the bijection in Theorem~\ref{treePathBijection} and the translation between contractions of $\nu$-Schr\"oder trees and contractions of $\nu$-Schr\"oder paths above, the next theorem now follows. 

\begin{theorem} \label{treePosetIsPathPoset}
The poset of $\nu$-Schr\"oder trees is isomorphic to the poset of $\nu$-Schr\"oder paths. 
\qed
\end{theorem}

\section{The face poset of the $\nu$-associahedron}
\label{sec.nuAssociahedron}

The $\nu$-associahedron $A_\nu$ is a polyhedral complex which generalizes the classical associahedron. 
It was introduced by Ceballos, Padrol and Sarmiento~\cite{CPS19}, and they gave a geometric realization of $A_\nu$ via tropical hyperplane arrangements. 
$A_\nu$ also has a combinatorial definition~\cite[Theorem 5.2]{CPS19} as a polyhedral complex whose face poset is determined by objects known as covering $(I,\overline{J})$-forests. 

In this section, we show that the face poset of the $\nu$-associahedron has alternative descriptions as a poset on $\nu$-Schr\"oder trees and as a poset of $\nu$-Schr\"oder paths by showing that these posets are isomorphic to the poset of covering $(I,\overline{J})$-forests.
We begin by recalling the definition of the covering $(I,\overline{J})$-forests of \cite{CPS19}, and for our purposes it suffices to restrict the definition slightly to set partitions of $[n]$.
\begin{definition}
Let $I\sqcup \overline{J}$ be a partition of $[n]$ such that $1 \in I$ and $n \in \overline{J}$. 
An {\em $(I,\overline{J})$-forest} is a subgraph of the complete bipartite graph $K_{|I|,|\overline{J}|}$ that is 
\begin{itemize}
    \item[1.] \textbf{Increasing:} each arc $(i,\overline{j})$ fulfills $i < \overline{j}$; and 
    \item[2.] \textbf{Non-crossing:} it does not contain two arcs $(i,\overline{j})$ and $(i',\overline{j}')$ satisfying $i < i' < j < \overline{j}'$. 
\end{itemize}
An {\em $(I,\overline{J})$-tree} is a maximal $(I,\overline{J})$-forest. 
A {\em covering $(I,\overline{J})$-forest} is an $(I,\overline{J})$-forest with the arc $(1, n)$ and no isolated nodes.
\end{definition}

To a set of covering $(I,\overline{J})$-forests we can associate a unique path $\nu$ as follows. 
Assign the label $E_{i-1}$ to the $i$-th element in $I$, and assign the label $N_{i-1}$ to the $i$-th element in $\overline{J}$. 
Reading the labels of the nodes $k=2,\ldots,n-1$ in increasing order yields a lattice path $\nu$ from $(0,0)$ to $(|I|-1,|\overline{J}|-1)$. 
See Figure~\ref{IJTreeAndNuSchroderTree} for an illustration.

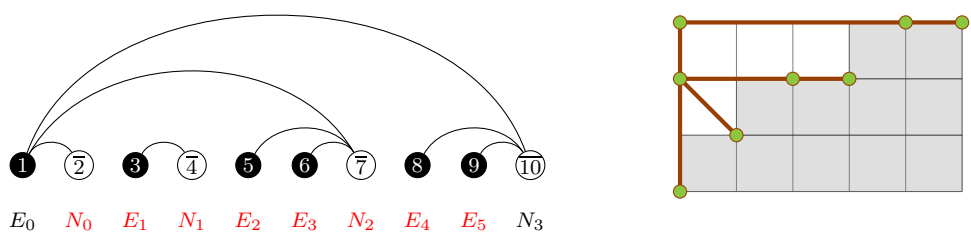
\begin{figure}[ht!]
\begin{tikzpicture}
\begin{scope}[xshift=0, yshift=0, scale=0.75]
	\node[style={circle,draw, inner sep=1pt, fill=black}] (1)  at (1,0)  {{\tiny \textcolor{white}{1}}};
	\node[style={circle,draw, inner sep=1pt, fill=black}] (3)  at (3,0)  {{\tiny \textcolor{white}{3}}};
	\node[style={circle,draw, inner sep=1pt, fill=black}] (5)  at (5,0)  {{\tiny \textcolor{white}{5}}};
	\node[style={circle,draw, inner sep=1pt, fill=black}] (6)  at (6,0)  {{\tiny \textcolor{white}{6}}};
	\node[style={circle,draw, inner sep=1pt, fill=black}] (8)  at (8,0)  {{\tiny \textcolor{white}{8}}};
	\node[style={circle,draw, inner sep=1pt, fill=black}] (9)  at (9,0)  {{\tiny \textcolor{white}{9}}};
	
	\node[style={circle,draw, inner sep=1pt, fill=none}] (2)  at (2,0)  {\tiny{$\overline{2}$}};
	\node[style={circle,draw, inner sep=1pt, fill=none}] (4)  at (4,0)  {{\tiny $\overline{4}$}};
	\node[style={circle,draw, inner sep=1pt, fill=none}] (7)  at (7,0)  {{\tiny $\overline{7}$}};
	\node[style={circle,draw, inner sep=0pt, fill=none}] (10)  at (10,0)  {{\tiny $\overline{10}$}};	
	
	\node[] (E0)  at (1,-1)  {\tiny $E_0$};
	\node[] (N0)  at (2,-1)  {\tiny \textcolor{red}{$N_0$}};
	\node[] (E1)  at (3,-1)  {\tiny \textcolor{red}{$E_1$}};
	\node[] (N1)  at (4,-1)  {\tiny \textcolor{red}{$N_1$}};	
	\node[] (E2)  at (5,-1)  {\tiny \textcolor{red}{$E_2$}};
	\node[] (E3)  at (6,-1)  {\tiny \textcolor{red}{$E_3$}};	
	\node[] (N2)  at (7,-1)  {\tiny \textcolor{red}{$N_2$}};	
	\node[] (E4)  at (8,-1)  {\tiny \textcolor{red}{$E_4$}};
	\node[] (E5)  at (9,-1)  {\tiny \textcolor{red}{$E_5$}};	
	\node[] (N3)  at (10,-1)  {\tiny $N_3$};	
	
	\draw[] (1) to [bend left=60] (2);
    \draw[] (1) to [bend left=60] (7);  
    \draw[] (1) to [bend left=70] (10);

    \draw[] (3) to [bend left=60] (4);

    \draw[] (5) to [bend left=60] (7);
    \draw[] (6) to [bend left=60] (7);
    
    \draw[] (8) to [bend left=60] (10);
    \draw[] (9) to [bend left=60] (10);
\end{scope}
\begin{scope}[xshift=270, yshift=-10, scale=0.75]
    \draw[help lines] (0,0) grid (5,3);
    \draw[fill=gray, opacity=0.25] (0,0) rectangle (5,1); 
    \draw[fill=gray, opacity=0.25] (1,1) rectangle (5,2);
    \draw[fill=gray, opacity=0.25] (3,2) rectangle (5,3);

    \draw[ultra thick, RawSienna] (0,0) -- (0,3);
    \draw[ultra thick, RawSienna] (1,1) -- (0,2);    
    \draw[ultra thick, RawSienna] (0,3) -- (5,3);    
    \draw[ultra thick, RawSienna] (0,2) -- (3,2);
    
	\vertex[draw=RawSienna, fill=LimeGreen] at (0,3) {};
	\vertex[draw=RawSienna, fill=LimeGreen] at (0,2) {};
	\vertex[draw=RawSienna, fill=LimeGreen] at (0,0) {};
	\vertex[draw=RawSienna, fill=LimeGreen] at (1,1) {};
	\vertex[draw=RawSienna, fill=LimeGreen] at (2,2) {};
	\vertex[draw=RawSienna, fill=LimeGreen] at (4,3) {};
	\vertex[draw=RawSienna, fill=LimeGreen] at (5,3) {};	
	\vertex[draw=RawSienna, fill=LimeGreen] at (3,2) {};
\end{scope}
\end{tikzpicture}
\caption{On the left is a covering $(I,\overline{J}$)-forest $F$ for $I=\{1,3,5,6,8,9\}$ and $\overline{J} = \{\overline{2},\overline{4},\overline{7},\overline{10}\}$. 
The associated path $\nu$ is read from the red labels below the covering $(I,\overline{J})$-forest.
On the right is the $\nu$-Schr\"oder tree that corresponds to $F$ under the bijection of Theorem~\ref{thm.coveringforesttreebijection}. }
\label{IJTreeAndNuSchroderTree}
\end{figure}

\begin{theorem} \label{thm.coveringforesttreebijection}
Covering $(I,\overline{J})$-forests are in bijection with $\nu$-Schr\"oder trees. 
\end{theorem}
\begin{proof}
Given a covering $(I,\overline{J})$-forest $F$, the arcs of $F$ can be identified with the labels at their end points, that is, pairs of the form $(E_i,N_j)$. 
For each such arc, insert a node at the coordinate $(i,j)$ of the grid from $(0,0)$ to $(|I|-1, |\overline{J}|-1)$, and call the resulting configuration of nodes in the grid $T$. 
The fact that $F$ has no isolated nodes guarantees that each row and column of the grid contains a node of $T$. 
The increasing condition guarantees that the nodes are in $R_\nu$, and the non-crossing condition guarantees that the nodes in $T$ are $\nu$-compatible. 
Thus $T$ is a $\nu$-Schr\"oder tree. 
This construction is readily invertible.  
\end{proof}

\begin{corollary}
Covering $(I,\overline{J})$-forests are in bijection with $\nu$-Schr\"oder paths. \qed
\end{corollary}

\begin{remark}
The $\nu$-Schr\"oder trees are thus grid representations of covering $(I,\overline{J})$-forests, just as $\nu$-binary trees are grid representations of $(I,\overline{J})$-trees in \cite[Remark 3.7]{CPS20}.
\end{remark}

\begin{definition}
The {\em poset of covering $(I,\overline{J})$-forests} is the set of covering $(I,\overline{J})$-forests equipped with the partial order $T \leq T'$ if and only if the arcs of $T'$ are a subset of the arcs of $T$.
Note that $T$ is covered by $T'$ if $T'$ has all but one of the arcs of $T$.
\end{definition}

Recall that for a polyhedral complex $\mathcal{C}$, the \textit{face poset} of $\mathcal{C}$ is the poset of non-empty faces of $\mathcal{C}$ with partial order $F_1 \leq F_2$ if and only if $F_1 \subseteq F_2$. 
The combinatorial definition of the $\nu$-associahedron is then given as follows.

\begin{definition}
Let $\nu$ be the lattice path associated with the set of the covering $(I,\overline{J})$-forests. 
The {\em $\nu$-associahedron} is the polyhedral complex whose face poset is the poset of covering $(I,\overline{J})$-forests. 
\end{definition}

We can combinatorially describe the $\nu$-associahedron in terms of $\nu$-Schr\"oder objects. 

\begin{theorem}
\label{thm:isoposets} The following posets are isomorphic: 
\begin{itemize}
    \item[1.] The face poset of the $\nu$-associahedron.
    \item[2.] The poset of $\nu$-Schr\"oder trees.
    \item[3.] The poset of $\nu$-Schr\"oder paths.
\end{itemize} 
\end{theorem}
\begin{proof}
Posets 1 and 2 are seen to be isomorphic since the cover relation in the poset of covering $(I,\overline{J})$-forests is equivalent to contracting the corresponding node in the $\nu$-Schr\"oder tree. 
The isomorphism between posets 2 and 3 was shown in Theorem~\ref{treePosetIsPathPoset}. 
\end{proof}

\begin{corollary}\label{cor:isoposets}
The number of $i$-dimensional faces of the $\nu$-associahedron is the number of $\nu$-Schr\"oder paths with $i$ diagonal steps, and therefore, the $\nu$-Schr\"oder numbers $\ttsch_\nu(i)$ enumerate the faces of $\nu$-associahedra.
\qed
\end{corollary}

A lattice is {\em Eulerian} if every nontrivial interval has an equal number of elements in the even ranks versus the odd ranks.

\begin{theorem}
Let $\widehat{P}$ denote the poset $P$ with an adjoined minimal element $\hat0$ and maximal element $\hat1$. 
If $P_\nu$ is the poset of $\nu$-Schr\"oder paths or $\nu$-Schr\"oder trees, then $\widehat{P}_\nu$ is a lattice. Furthermore, every interval $[x,y]$ in $\widehat{P}_\nu\setminus \{\hat{1}\}$ is an Eulerian lattice.  
\label{thm:lattice}
\end{theorem}

\begin{proof}
Since $A_\nu$ is a polytopal complex, $P\cup \{\hat{0}\}$ is a meet semilattice, with the meet of two faces being their (possibly empty) intersection. 
Since $z \wedge_{\widehat{P}} \{\hat{1}\} = z$, $\widehat{P}$ is a meet semilattice. 

Let $x,y\in \widehat{P}$. If there exists upper bounds $z, w \in \widehat{P}$ of both $x$ and $y$, that is, $z$ and $w$ satisfy $x <_{\widehat{P}} z$, $y <_{\widehat{P}} z$, $x <_{\widehat{P}} w$ and $y <_{\widehat{P}} w$. Then the unique face at the intersection of the faces $w$ and $z$ is the unique join $x\vee_{\widehat{P}} y$. If there is no face containing $x$ and $y$ as subfaces in $A_\nu$, then $x\vee_{\widehat{P}} y = \hat{1}$. Every interval $[\hat{0},y] \in \widehat{P}\setminus \{\hat{1}\}$ corresponds to a convex polytope in $A_\nu$, and hence is Eulerian. Therefore every subinterval $[x,y] \subseteq [\hat{0},y]$ in $\widehat{P}\setminus \{\hat{1}\}$ is an Eulerian lattice.
\end{proof}

Let $\nu$ be a lattice path with $n$ steps. 
Pr\'eville-Ratelle and Viennot~\cite[Theorem 3]{PV17} showed that the $\nu$-Tamari lattice is isomorphic to an interval in the classical $(NE)^{n+1}$ Tamari lattice. 
Extending this isomorphism gives that the $\nu$-associahedron is isomorphic to a connected subcomplex of the boundary complex of the $n$-associahedron. As a result, we have the following corollary. 

\begin{corollary}
If $\nu = (EN)^{n+1}$, then $P_\nu \cup \{\hat{0}\} $ is isomorphic to the face lattice of the $n$-associahedron. For general $\nu$, $\widehat{P}_\nu$ is isomorphic to a sublattice of the face lattice of the $m$-associahedron, where $m$ is the number of steps in $\nu$.\qed 
\label{cor:sublattice}
\end{corollary}

Since the classical Tamari lattice can be partitioned into disjoint intervals of $\nu$-Tamari lattices \cite[Theorem 3]{PV17}, another consequence is that
$$ \bigcup_{\substack{\nu \text{ path of} \\ \text{length }n }} \widehat{P}_\nu \cong F$$
where $F$ is a sublattice of the face lattice of the $n$-associahedron.

For our last result, we apply discrete Morse theory to the contraction poset of $\nu$-Schr\"oder paths to show that the $\nu$-associahedron $A_\nu$ is contractible. 
See~\cite{Koz08} for background on discrete Morse theory.  

\begin{definition}
Given a poset $P$, a \textit{partial matching} in $P$ is a matching in the underlying graph of the Hasse diagram of $P$. 
That is, a subset $M\subseteq P\times P$, such that 
\begin{itemize}
    \item $(a,b)\in M$ implies $a\prec b$;
    \item each element $a\in P$ belongs to at most one element of $M$.
\end{itemize}
When $(a,b) \in M$, we write $a=d(b)$ and $b = u(a)$. 
A partial matching is {\em acyclic} if there does not exist a cycle 
$$b_1  \succ d(b_1) \prec b_2 \succ d(b_2) \prec \cdots \prec b_n \succ d(b_n) \prec b_1$$
where $n\geq 2$ and the $b_i \in P$ are distinct. 
Any elements of $P$ not in an element of $M$ are called {\em critical} elements.
\label{def:acyclicMatching}
\end{definition}

The main theorem of discrete Morse theory for complexes is the following. 
\begin{theorem}[{\cite[Theorem 11.13]{Koz08}}] \label{thm:morse} Let $\mathcal{C}$ be a polyhedral complex with face poset $F$. 
Let $M$ be an acyclic matching on $F$, and let $c_i$ denote the number of critical elements in $F$ corresponding to $i$-dimensional faces of $\mathcal{C}$. Then $\mathcal{C}$ is homotopy equivalent to a subcomplex of $\mathcal{C}$ consisting of $c_i$ faces of dimension $i$. 
\end{theorem}

\begin{theorem}\label{thm:contractible}
The $\nu$-associahedron $A_\nu$ is contractible. 
\end{theorem}

\begin{proof}
Let $A_\nu$ be the $\nu$-associahedron with face poset $P_\nu$. 
By Theorems~\ref{thm:isoposets} and~\ref{thm:morse}, it suffices to find an acyclic matching on $P_\nu$ with a single critical element corresponding to a vertex in $A_\nu$. 

Let $M$ be the set of edges $(\pi,\sigma)$ where $\pi$ is formed from $\sigma$ by replacing with $EN$ the first $D$ step not preceded by any valley. We claim that $M$ is the desired acyclic partial matching. 
See Figure~\ref{fig:53poset} for an example.

First we check that $M$ is in fact a partial matching. 
If $(\pi,\sigma) \in M$, then $\sigma$ is formed by a contraction of $\pi$, so $\pi\prec \sigma$ in $P_\nu$. 
Next we show that a path $\pi$ cannot be in more than one element of $M$. 
First, there cannot be a pair of elements  $(\tau,\pi)$ and $(\pi,\sigma)$ in $M$ because all $D$ steps in $\pi=d(\sigma)$ are preceded by the added valley and so $\tau=d(\pi)$ cannot exist.
Second, since $d(\pi)$ is unique by construction, it follows that there cannot be two pairs $(\tau,\pi)$ and $(\tau',\pi)$ in $M$ where $\tau\neq \tau'$. 
It remains to check that there are no two pairs $(\pi,\sigma)$ and $(\pi,\rho)$ in $M$ with $\sigma\neq \rho$. Suppose the contrary, then $\sigma$ and $\rho$ can be partitioned into $\sigma = \sigma_1D_\sigma\sigma_2$ and $\rho = \rho_1D_\rho \rho_2$, where the $D_\sigma$ and $D_\rho$ steps are the first $D$ steps not preceded by a valley in the respective paths $\sigma$ and $\rho$. If $\sigma_1$ and $\rho_1$ have the same number of steps, then it follows from $\pi = \sigma_1EN\sigma_2 = \rho_1EN\rho_2$ that $\sigma_1 = \rho_1$. However, we cannot have $\sigma_1= \rho_1$, because then we would also have $\sigma_2 = \rho_2$, from which it would follow that $\sigma = \sigma_1D_\sigma\sigma_2 = \rho_1D_\rho \rho_2 = \rho$. Thus either $\sigma_1$ has fewer steps than $\rho_1$ or vice versa. If $\sigma_1$ has fewer steps, then $\pi$ can be partitioned as $\pi = \pi_1 EN \pi_2 EN \rho_2$, where $\pi_1EN\pi_2 = \rho_1$. However, this means $\rho = \pi_1EN\pi_2D_\rho \rho_2$ has a valley before $D_\rho$, which contradicts the fact that $(\pi, \rho)$ is in $M$. Similarly $\rho_1$ cannot have fewer steps. 
We conclude that $M$ is a partial matching. 

Next, we check that $M$ is acyclic. Suppose to the contrary that there exists a cycle 
$$\pi_1  \succ d(\pi_1) \prec \pi_2 \succ d(\pi_2) \prec \cdots \prec \pi_n \succ d(\pi_n) \prec \pi_1$$
with $n\geq 2$. 
Note that any pair $(d(\pi_i), \pi_i)$ satisfies $\mathrm{area}(d(\pi_i))=\mathrm{area}(\pi_i) - 1/2 $.
Every pair $d(\pi_i) \prec \pi_j$ in the cycle is related by a contraction of $d(\pi_i)$, and each contraction move either decreases the area of the path, or adds exactly half a unit of area. 
Since $\mathrm{area}(\pi_1)$ at the beginning and the end of the cycle must be equal, each contraction between $d(\pi_i)$ and $\pi_j$ must increase the area by exactly one half, and must therefore be a right contraction. 
The first valley in $d(\pi_1)$ is the one added to $\pi_1$. Since $\pi_2$ must have a $D$ step not preceded by a valley, it must be a result of a right contraction at the first $EN$ pair in $d(\pi_1)$, which means $\pi_1 = \pi_2$. 
Therefore $n<2$, giving the desired contradiction, and so $M$ is acyclic.

Finally, we check that the only critical element in $P_\nu$ is the path $N^aE^b$. Any other path $\pi$ will have either a first $D$ step not preceded by a valley, or not. If it does, then $(d(\pi),\pi) \in M$. If it does not have such a $D$, step, then it must have a first valley. Letting $\sigma$ be the path $\pi$ but with the first valley replaced with a $D$ step gives an element $(\pi, \sigma) \in M$. 
\end{proof}

\begin{remark}
The acyclic matching $M$ is more difficult to describe in the setting of $\nu$-Schr\"oder trees or of $(I,\overline{J})$-trees, thus highlighting a benefit of the $\nu$-Schr\"oder path perspective. 
The utility of paths is the clear linear order on the steps, making it easy to check if valleys occur before a $D$ step.
\end{remark}

\makeatletter
\tikzset{%
  prefix node name/.code={%
    \tikzset{%
      name/.code={\edef\tikz@fig@name{#1 ##1}}
    }%
  }%
}
\makeatother
\begin{figure}[ht!]
\begin{tikzpicture}[scale=1.2]
\tikzstyle{vertex}=[circle, inner sep=0pt, minimum size=0pt] 
\begin{scope}[scale=.25, xshift=270, yshift=500, prefix node name=C1]
	\draw[fill, color=gray!25] (0,0) rectangle (5,1);
	\draw[fill, color=gray!25] (1,1) rectangle (5,2);
	\draw[fill, color=gray!25] (3,2) rectangle (5,3);
	\draw[very thin, color=gray!100] (0,0) grid (5,3);		
	\draw[very thick, color=red] (0,0)--(0,1)--(1,2)--(2,2)--(3,3)--(5,3);
	\node[vertex](b) at (2.5,-0.5){};
\end{scope}
\begin{scope}[scale=.25, xshift=770, yshift=500, prefix node name=C2]
	\draw[fill, color=gray!25] (0,0) rectangle (5,1);
	\draw[fill, color=gray!25] (1,1) rectangle (5,2);
	\draw[fill, color=gray!25] (3,2) rectangle (5,3);
	\draw[very thin, color=gray!100] (0,0) grid (5,3);		
	\draw[very thick, color=red] (0,0)--(0,1)--(2,3)--(5,3);
	\node[vertex](b) at (2.5,-0.5){};
\end{scope}
\begin{scope}[scale=.25, xshift=-90, yshift=250, prefix node name=B1]
	\draw[fill, color=gray!25] (0,0) rectangle (5,1);
	\draw[fill, color=gray!25] (1,1) rectangle (5,2);
	\draw[fill, color=gray!25] (3,2) rectangle (5,3);
	\draw[very thin, color=gray!100] (0,0) grid (5,3);		
	\draw[very thick, color=red] (0,0)--(0,1)--(1,2)--(3,2)--(3,3)--(5,3);
	\node[vertex](t) at (2.5,3.5){};
	\node[vertex](b) at (2.5,-0.5){};
\end{scope}
\begin{scope}[scale=.25, xshift=90, yshift=250, prefix node name=B2]
	\draw[fill, color=gray!25] (0,0) rectangle (5,1);
	\draw[fill, color=gray!25] (1,1) rectangle (5,2);
	\draw[fill, color=gray!25] (3,2) rectangle (5,3);
	\draw[very thin, color=gray!100] (0,0) grid (5,3);		
	\draw[very thick, color=red] (0,0)--(0,1)--(1,1)--(1,2)--(2,2)--(3,3)--(5,3);
	\node[vertex](t) at (2.5,3.5){};
	\node[vertex](b) at (2.5,-0.5){};
\end{scope}
\begin{scope}[scale=.25, xshift=270, yshift=250, prefix node name=B3]
	\draw[fill, color=gray!25] (0,0) rectangle (5,1);
	\draw[fill, color=gray!25] (1,1) rectangle (5,2);
	\draw[fill, color=gray!25] (3,2) rectangle (5,3);
	\draw[very thin, color=gray!100] (0,0) grid (5,3);		
	\draw[very thick, color=red] (0,0)--(0,2)--(2,2)--(3,3)--(5,3);
	\node[vertex](t) at (2.5,3.5){};
	\node[vertex](b) at (2.5,-0.5){};
\end{scope}
\begin{scope}[scale=.25, xshift=630, yshift=250, prefix node name=B4]
	\draw[fill, color=gray!25] (0,0) rectangle (5,1);
	\draw[fill, color=gray!25] (1,1) rectangle (5,2);
	\draw[fill, color=gray!25] (3,2) rectangle (5,3);
	\draw[very thin, color=gray!100] (0,0) grid (5,3);		
	\draw[very thick, color=red] (0,0)--(0,1)--(1,2)--(2,2)--(2,3)--(5,3);
	\node[vertex](t) at (2.5,3.5){};
	\node[vertex](b) at (2.5,-0.5){};
\end{scope}
\begin{scope}[scale=.25, xshift=450, yshift=250, prefix node name=B5]
	\draw[fill, color=gray!25] (0,0) rectangle (5,1);
	\draw[fill, color=gray!25] (1,1) rectangle (5,2);
	\draw[fill, color=gray!25] (3,2) rectangle (5,3);
	\draw[very thin, color=gray!100] (0,0) grid (5,3);		
	\draw[very thick, color=red]
	(0,0)--(0,2)--(1,2)--(2,3)--(5,3);
	\node[vertex](t) at (2.5,3.5){};
	\node[vertex](b) at (2.5,-0.5){};
\end{scope}
\begin{scope}[scale=.25, xshift=810, yshift=250, prefix node name=B6]
	\draw[fill, color=gray!25] (0,0) rectangle (5,1);
	\draw[fill, color=gray!25] (1,1) rectangle (5,2);
	\draw[fill, color=gray!25] (3,2) rectangle (5,3);
	\draw[very thin, color=gray!100] (0,0) grid (5,3);		
	\draw[very thick, color=red] (0,0)--(0,2)--(1,3)--(5,3);
	\node[vertex](t) at (2.5,3.5){};
	\node[vertex](b) at (2.5,-0.5){};
\end{scope}
\begin{scope}[scale=.25, xshift=990, yshift=250, prefix node name=B7]
	\draw[fill, color=gray!25] (0,0) rectangle (5,1);
	\draw[fill, color=gray!25] (1,1) rectangle (5,2);
	\draw[fill, color=gray!25] (3,2) rectangle (5,3);
	\draw[very thin, color=gray!100] (0,0) grid (5,3);		
	\draw[very thick, color=red] (0,0)--(0,1)--(1,2)--(1,3)--(5,3);
	\node[vertex](t) at (2.5,3.5){};
	\node[vertex](b) at (2.5,-0.5){};	
\end{scope}
\begin{scope}[scale=.25, xshift=1170, yshift=250, prefix node name=B8]
	\draw[fill, color=gray!25] (0,0) rectangle (5,1);
	\draw[fill, color=gray!25] (1,1) rectangle (5,2);
	\draw[fill, color=gray!25] (3,2) rectangle (5,3);
	\draw[very thin, color=gray!100] (0,0) grid (5,3);		
	\draw[very thick, color=red] (0,0)--(0,1)--(1,1)--(1,2)--(2,3)--(5,3);	
	\node[vertex](t) at (2.5,3.5){};
	\node[vertex](b) at (2.5,-0.5){};
\end{scope}
\begin{scope}[scale=.25, prefix node name=A1]
	\draw[fill, color=gray!25] (0,0) rectangle (5,1);
	\draw[fill, color=gray!25] (1,1) rectangle (5,2);
	\draw[fill, color=gray!25] (3,2) rectangle (5,3);
	\draw[very thin, color=gray!100] (0,0) grid (5,3);		
	\draw[very thick, color=red] (0,0)--(0,1)--(1,1)--(1,2)--(3,2)--(3,3)--(5,3);
	\node[vertex](t) at (2.5,3.5){};
\end{scope}
\begin{scope}[scale=.25, xshift=180, prefix node name=A2]
	\draw[fill, color=gray!25] (0,0) rectangle (5,1);
	\draw[fill, color=gray!25] (1,1) rectangle (5,2);
	\draw[fill, color=gray!25] (3,2) rectangle (5,3);
	\draw[very thin, color=gray!100] (0,0) grid (5,3);		
	\draw[very thick, color=red] (0,0)--(0,2)--(3,2)--(3,3)--(5,3);
	\node[vertex](t) at (2.5,3.5){};
\end{scope}
\begin{scope}[scale=.25, xshift=540, prefix node name=A3]
	\draw[fill, color=gray!25] (0,0) rectangle (5,1);
	\draw[fill, color=gray!25] (1,1) rectangle (5,2);
	\draw[fill, color=gray!25] (3,2) rectangle (5,3);
	\draw[very thin, color=gray!100] (0,0) grid (5,3);		
	\draw[very thick, color=red] (0,0)--(0,1)--(1,1)--(1,2)--(2,2)--(2,3)--(5,3);
	\node[vertex](t) at (2.5,3.5){};
\end{scope}
\begin{scope}[scale=.25, xshift=360, prefix node name=A4]
	\draw[fill, color=gray!25] (0,0) rectangle (5,1);
	\draw[fill, color=gray!25] (1,1) rectangle (5,2);
	\draw[fill, color=gray!25] (3,2) rectangle (5,3);
	\draw[very thin, color=gray!100] (0,0) grid (5,3);		
	\draw[very thick, color=red] (0,0)--(0,2)--(2,2)--(2,3)--(5,3);
	\node[vertex](t) at (2.5,3.5){};
\end{scope}
\begin{scope}[scale=.25, xshift=720, prefix node name=A5]
	\draw[fill, color=gray!25] (0,0) rectangle (5,1);
	\draw[fill, color=gray!25] (1,1) rectangle (5,2);
	\draw[fill, color=gray!25] (3,2) rectangle (5,3);
	\draw[very thin, color=gray!100] (0,0) grid (5,3);		
	\draw[very thick, color=red] (0,0)--(0,2)--(1,2)--(1,3)--(5,3);
	\node[vertex](t) at (2.5,3.5){};
\end{scope}
\begin{scope}[scale=.25, xshift=900, prefix node name=A6]
	\draw[fill, color=gray!25] (0,0) rectangle (5,1);
	\draw[fill, color=gray!25] (1,1) rectangle (5,2);
	\draw[fill, color=gray!25] (3,2) rectangle (5,3);
	\draw[very thin, color=gray!100] (0,0) grid (5,3);		
	\draw[very thick, color=red] (0,0)--(0,3)--(5,3);
	\node[vertex](t) at (2.5,3.5){};		
\end{scope}
\begin{scope}[scale=.25, xshift=1080, prefix node name=A7]
	\draw[fill, color=gray!25] (0,0) rectangle (5,1);
	\draw[fill, color=gray!25] (1,1) rectangle (5,2);
	\draw[fill, color=gray!25] (3,2) rectangle (5,3);
	\draw[very thin, color=gray!100] (0,0) grid (5,3);		
	\draw[very thick, color=red] (0,0)--(0,1)--(1,1)--(1,3)--(5,3);
	\node[vertex](t) at (2.5,3.5){};	
\end{scope}

\draw[ultra thin] (A1 t)--(B2 b);
\draw[ultra thin] (A2 t)--(B1 b);
\draw[ultra thin] (A3 t)--(B2 b); \draw[ultra thin] (A3 t)--(B8 b);
\draw[ultra thin] (A4 t)--(B3 b); 
\draw[ultra thin] (A5 t)--(B4 b); \draw[ultra thin] (A5 t)--(B5 b);
\draw[ultra thin] (A6 t)--(B6 b); \draw[ultra thin] (A6 t)--(B7 b);
\draw[ultra thin] (A7 t)--(B8 b);
\draw[ultra thin] (B1 t)--(C1 b); 
\draw[ultra thin] (B3 t)--(C1 b); 
\draw[ultra thin] (B4 t)--(C1 b);
\draw[ultra thin] (B4 t)--(C2 b); 
\draw[ultra thin] (B5 t)--(C1 b);
\draw[ultra thin] (B6 t)--(C2 b); 
\draw[ultra thin] (B7 t)--(C2 b);

\draw[ultra thick,color=Cerulean] (A1 t)--(B1 b);
\draw[ultra thick,color=Cerulean] (A2 t)--(B3 b);
\draw[ultra thick,color=Cerulean] (A3 t)--(B4 b);
\draw[ultra thick,color=Cerulean] (A4 t)--(B5 b);
\draw[ultra thick,color=Cerulean] (A5 t)--(B6 b);
\draw[ultra thick, color=Cerulean] (A7 t)--(B7 b);
\draw[ultra thick,color=Cerulean] (B2 t)--(C1 b);
\draw[ultra thick,color=Cerulean] (B8 t)--(C2 b); 

\end{tikzpicture}
\caption{The contraction poset of $(3,5)$-Schr\"oder paths, which is the face poset of the $(3,5)$-associahedron of Figure~\ref{fig.35associahedron}. 
The blue edges denote the acyclic partial matching $M$ described in the proof of Theorem~\ref{thm:contractible}.
The path $N^3E^5$ is the unique critical element in this matching.}
\label{fig:53poset}
\end{figure}
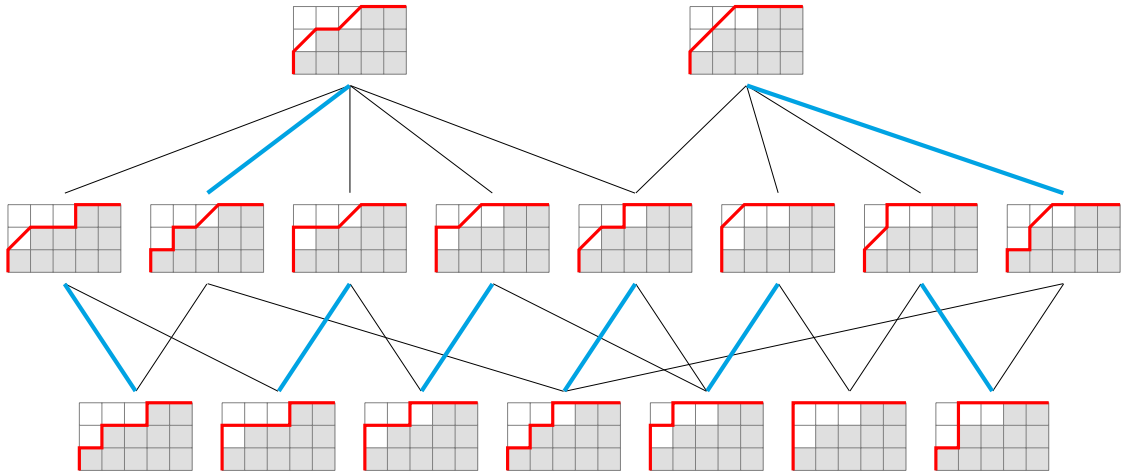

\begin{figure}[ht!]
\begin{tikzpicture}[scale=.75]
\tikzstyle{vertex}=[circle, fill=black, inner sep=0pt, minimum size=5pt]
\begin{scope}[scale=.6, xshift=0, yshift=0]
	\vertex at (0,0)  {};
	\vertex at (5,4) {};
	\vertex at (10,8) {};
	\vertex at (6,13) {};
	\vertex at (.5,9.7) {};
	\vertex at (-4,7) {};
	\vertex at (-4,3) {};
	\draw[thick] (0,0) -- (10,8) -- (6,13) -- (-4,7) -- (-4,3) -- (0,0);
	\draw[thick] (5,4) -- (.5,9.7);
\end{scope}	

\begin{scope}[scale=0.2, xshift=-65, yshift=350]
	\draw[fill, color=gray!25] (0,0) rectangle (5,1);
	\draw[fill, color=gray!25] (1,1) rectangle (5,2);
	\draw[fill, color=gray!25] (3,2) rectangle (5,3);
	\draw[very thin, color=gray!100] (0,0) grid (5,3);		
	\draw[thick, color=red] (0,0)--(0,1)--(1,2)--(2,2)--(3,3)--(5,3);
\end{scope}

\begin{scope}[scale=0.2, xshift=400, yshift=680]
	\draw[fill, color=gray!25] (0,0) rectangle (5,1);
	\draw[fill, color=gray!25] (1,1) rectangle (5,2);
	\draw[fill, color=gray!25] (3,2) rectangle (5,3);
	\draw[very thin, color=gray!100] (0,0) grid (5,3);		
	\draw[thick, color=red] (0,0)--(0,1)--(2,3)--(3,3)--(5,3);
\end{scope}

\begin{scope}[scale=0.2, xshift=70, yshift=600]
	\draw[fill, color=gray!25] (0,0) rectangle (5,1);
	\draw[fill, color=gray!25] (1,1) rectangle (5,2);
	\draw[fill, color=gray!25] (3,2) rectangle (5,3);
	\draw[very thin, color=gray!100] (0,0) grid (5,3);		
	\draw[thick, color=red] (0,0)--(0,1)--(1,2)--(2,2) -- (2,3) --(5,3);
\end{scope}

\begin{scope}[scale=0.2, xshift=-280, yshift=5]
	\draw[fill, color=gray!25] (0,0) rectangle (5,1);
	\draw[fill, color=gray!25] (1,1) rectangle (5,2);
	\draw[fill, color=gray!25] (3,2) rectangle (5,3);
	\draw[very thin, color=gray!100] (0,0) grid (5,3);		
	\draw[thick, color=red] (0,0)--(0,1)--(1,2)--(3,2)--(3,3)--(5,3);
\end{scope}

\begin{scope}[scale=0.2, xshift=-500, yshift=380]
	\draw[fill, color=gray!25] (0,0) rectangle (5,1);
	\draw[fill, color=gray!25] (1,1) rectangle (5,2);
	\draw[fill, color=gray!25] (3,2) rectangle (5,3);
	\draw[very thin, color=gray!100] (0,0) grid (5,3);		
	\draw[thick, color=red] (0,0)--(0,2)--(2,2)--(3,3) --(5,3);
\end{scope}

\begin{scope}[scale=0.2, xshift=-300, yshift=720]
	\draw[fill, color=gray!25] (0,0) rectangle (5,1);
	\draw[fill, color=gray!25] (1,1) rectangle (5,2);
	\draw[fill, color=gray!25] (3,2) rectangle (5,3);
	\draw[very thin, color=gray!100] (0,0) grid (5,3);		
	\draw[thick, color=red] (0,0)--(0,2)--(1,2)--(2,3) --(5,3);
\end{scope}

\begin{scope}[scale=0.2, xshift=150, yshift=990]
	\draw[fill, color=gray!25] (0,0) rectangle (5,1);
	\draw[fill, color=gray!25] (1,1) rectangle (5,2);
	\draw[fill, color=gray!25] (3,2) rectangle (5,3);
	\draw[very thin, color=gray!100] (0,0) grid (5,3);		
	\draw[thick, color=red] (0,0)--(0,2)--(1,3) --(5,3);
\end{scope}

\begin{scope}[scale=0.2, xshift=690, yshift=910]
	\draw[fill, color=gray!25] (0,0) rectangle (5,1);
	\draw[fill, color=gray!25] (1,1) rectangle (5,2);
	\draw[fill, color=gray!25] (3,2) rectangle (5,3);
	\draw[very thin, color=gray!100] (0,0) grid (5,3);		
	\draw[thick, color=red] (0,0)--(0,1)--(1,2) --(1,3) -- (5,3);
\end{scope}

\begin{scope}[scale=0.2, xshift=650, yshift=420]
	\draw[fill, color=gray!25] (0,0) rectangle (5,1);
	\draw[fill, color=gray!25] (1,1) rectangle (5,2);
	\draw[fill, color=gray!25] (3,2) rectangle (5,3);
	\draw[very thin, color=gray!100] (0,0) grid (5,3);		
	\draw[thick, color=red] (0,0)--(0,1)--(1,1) --(1,2) -- (2,3) -- (5,3);
\end{scope}

\begin{scope}[scale=0.2, xshift=210, yshift=60]
	\draw[fill, color=gray!25] (0,0) rectangle (5,1);
	\draw[fill, color=gray!25] (1,1) rectangle (5,2);
	\draw[fill, color=gray!25] (3,2) rectangle (5,3);
	\draw[very thin, color=gray!100] (0,0) grid (5,3);		
	\draw[thick, color=red] (0,0)--(0,1)--(1,1) --(1,2) -- (2,2) -- (3,3) -- (5,3);
\end{scope}

\begin{scope}[scale=0.2, xshift=30, yshift=-130]
	\draw[fill, color=gray!25] (0,0) rectangle (5,1);
	\draw[fill, color=gray!25] (1,1) rectangle (5,2);
	\draw[fill, color=gray!25] (3,2) rectangle (5,3);
	\draw[very thin, color=gray!100] (0,0) grid (5,3);		
	\draw[thick, color=red] (0,0)--(0,1)--(1,1)--(1,2)--(3,2)--(3,3)--(5,3);
\end{scope}

\begin{scope}[scale=0.2, xshift=-540, yshift=170]
	\draw[fill, color=gray!25] (0,0) rectangle (5,1);
	\draw[fill, color=gray!25] (1,1) rectangle (5,2);
	\draw[fill, color=gray!25] (3,2) rectangle (5,3);
	\draw[very thin, color=gray!100] (0,0) grid (5,3);		
	\draw[thick, color=red] (0,0)--(0,2)--(3,2)--(3,3)--(5,3);
\end{scope}

\begin{scope}[scale=0.2, xshift=-540, yshift=600]
	\draw[fill, color=gray!25] (0,0) rectangle (5,1);
	\draw[fill, color=gray!25] (1,1) rectangle (5,2);
	\draw[fill, color=gray!25] (3,2) rectangle (5,3);
	\draw[very thin, color=gray!100] (0,0) grid (5,3);		
	\draw[thick, color=red] (0,0)--(0,2)-- (2,2) --(2,3) -- (3,3) --(5,3);
\end{scope}

\begin{scope}[scale=0.2, xshift=-100, yshift=890]
	\draw[fill, color=gray!25] (0,0) rectangle (5,1);
	\draw[fill, color=gray!25] (1,1) rectangle (5,2);
	\draw[fill, color=gray!25] (3,2) rectangle (5,3);
	\draw[very thin, color=gray!100] (0,0) grid (5,3);		
	\draw[thick, color=red] (0,0)--(0,2)-- (1,2) --(1,3) -- (3,3) --(5,3);
\end{scope}

\begin{scope}[scale=0.2, xshift=460, yshift=210]
	\draw[fill, color=gray!25] (0,0) rectangle (5,1);
	\draw[fill, color=gray!25] (1,1) rectangle (5,2);
	\draw[fill, color=gray!25] (3,2) rectangle (5,3);
	\draw[very thin, color=gray!100] (0,0) grid (5,3);		
	\draw[thick, color=red] (0,0)--(0,1)-- (1,1) --(1,2) -- (2,2) -- (2,3) -- (5,3);
\end{scope}

\begin{scope}[scale=0.2, xshift=920, yshift=600]
	\draw[fill, color=gray!25] (0,0) rectangle (5,1);
	\draw[fill, color=gray!25] (1,1) rectangle (5,2);
	\draw[fill, color=gray!25] (3,2) rectangle (5,3);
	\draw[very thin, color=gray!100] (0,0) grid (5,3);		
	\draw[thick, color=red] (0,0)--(0,1)-- (1,1) --(1,3) -- (5,3);
\end{scope}

\begin{scope}[scale=0.2, xshift=520, yshift=1150]
	\draw[fill, color=gray!25] (0,0) rectangle (5,1);
	\draw[fill, color=gray!25] (1,1) rectangle (5,2);
	\draw[fill, color=gray!25] (3,2) rectangle (5,3);
	\draw[very thin, color=gray!100] (0,0) grid (5,3);		
	\draw[thick, color=red] (0,0)--(0,3)-- (5,3);
\end{scope}

\end{tikzpicture}
\caption{The $(3,5)$-associahedron with its faces indexed by $(3,5)$-Schr\"oder paths.}
\label{fig.35associahedron}
\end{figure}
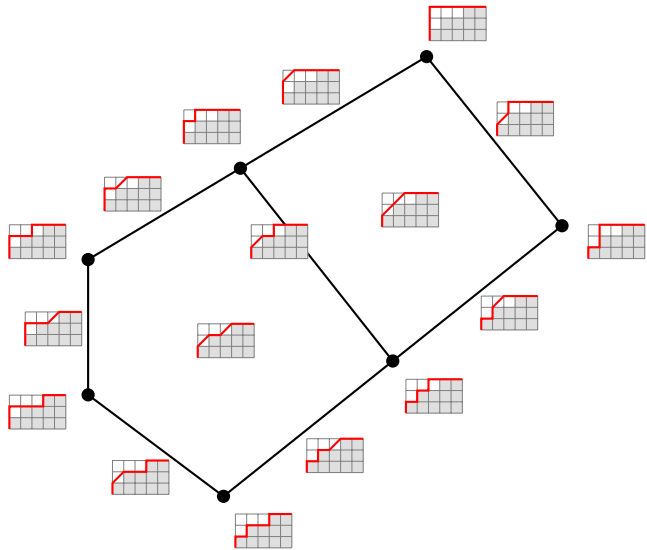

\bibliographystyle{plain}
\bibliography{schroeder}

\end{document}